\documentclass{article}
\usepackage[utf8]{inputenc}
\usepackage{subfiles}
\usepackage{algorithm}
\usepackage{amsmath,amsthm,amssymb,amsfonts, amsbsy, enumitem, fancyhdr, color, comment, graphicx, environ, wasysym, verbatim, bbm, graphics, hyperref, mathtools, float, upquote, mathrsfs, tikz-cd, tikz, changepage, parskip, mathrsfs, mdframed,tabularx,multirow, thmtools, easybmat, booktabs,xcolor}
\usepackage[T1]{fontenc}
\usetikzlibrary{arrows,matrix,positioning}
\usetikzlibrary{shadows}
\makeatletter
\newcommand{\multiline}[1]{%
  \begin{tabularx}{\dimexpr\linewidth-\ALG@thistlm}[t]{@{}X@{}}
    #1
  \end{tabularx}
}
\makeatother

\usepackage[normalem]{ulem}
\usepackage{algpseudocode}
\usepackage{array}
\usepackage{subcaption}
\usepackage{hyperref}
\let\temp\phi
\let\phi\varphi
\let\varphi\temp

\usepackage[paper=letterpaper,left=25mm,right=25mm,top=3cm,bottom=25mm]{geometry}

\newcolumntype{x}[1]{>{\centering\arraybackslash\hspace{0pt}}p{#1}}
\usepackage{parskip}
\usepackage{authblk}
\usepackage{titlefoot}

\theoremstyle{plain}
\newtheorem{theorem}{Theorem}
\newtheorem{problem}{Problem}
\newtheorem{definition}{Definition}
\newtheorem{proposition}{Proposition}
\newtheorem{lemma}{Lemma}
\newtheorem{conjecture}{Conjecture}
\declaretheorem[style=definition,name=Example,qed=$\square$]{example}

\title{Orthogonal Decomposition of Tensor Trains}
\author[]{Karim Halaseh}
\author[]{Tommi Muller}
\author[]{Elina Robeva}
\affil[]{\normalsize\textit{The University of British Columbia}}
\date{}

\begin{document}

\maketitle
\begin{abstract}
In this paper we study the problem of decomposing a given tensor into a tensor train such that the tensors at the vertices are orthogonally decomposable. When the tensor train has length two, and the orthogonally decomposable tensors at the two vertices are symmetric, we recover the decomposition by considering random linear combinations of slices. Furthermore, if the tensors at the vertices are symmetric and low-rank but not orthogonally decomposable, we show that a whitening procedure can transform the problem into the orthogonal case. When the tensor network has length three or more and the tensors at the vertices are symmetric and orthogonally decomposable, we provide an algorithm for recovering them subject to some rank conditions. Finally, in the case of tensor trains of length two in which the tensors at the vertices are orthogonally decomposable but not necessarily symmetric, we show that the decomposition problem reduces to the novel problem of decomposing a matrix into an orthogonal matrix multiplied by diagonal matrices on either side. We provide and compare two solutions, one based on Sinkhorn's theorem and one on Procrustes' algorithm. We conclude with a multitude of open problems in linear and multilinear algebra that arose in our study.

\end{abstract}

\unmarkedfntext{\noindent \hspace{-0.33cm} Keywords: Tensor decompositions, CP decomposition, Orthogonally decomposable tensors, Tensor networks, Tensor trains, Matrix decompositions, Sinkhorn's algorithm, Procrustes problem

MSC2020 Subject Classification: 15A69, 15A29, 15B10}


\section{Introduction}

\subfile{introduction}

\section{Background}\label{sec:background}

In this section we provide the necessary background and notation on tensors, tensor decompositions, tensor networks, and orthogonally decomposable tensors.
\subsection{Tensors}
We denote the set $\{1,...,n\}$ as $[n]$, for $n \in \mathbb{N}$. A $d$-tensor with  real entries is an element in $\mathbb R^{n_1\times\cdots\times n_d}$ for some $n_1,\ldots, n_d\in\mathbb N$. Scalars, vectors, and matrices are $0$-, $1$-, and $2$-tensors, respectively. We denote scalars by lowercase letters (e.g. $\lambda$), vectors by bolded lowercase letters (e.g. $\mathbf{u}$), matrices by bolded uppercase letters (e.g. $\mathbf{U}$), and $d$-tensors by script uppercase letters (e.g. $\mathcal{T}$), for $d \geq 3$. If $\mathbf{u}_1 \in \mathbb{R}^{n_1},...,\mathbf{u}_d \in \mathbb{R}^{n_d}$, then one can form their tensor  product, which is the $d$-tensor $\mathcal{T} = \mathbf{u}_1 \otimes ... \otimes \mathbf{u}_d \in \mathbb{R}^{n_1 \times ... \times n_d}$ with entries
$$\mathcal{T}_{i_1...i_d} = (\mathbf{u}_1)_{i_1}...(\mathbf{u}_d)_{i_d}, \hspace{1cm} i_j \in [n_j], \; j \in [d]$$
Such tensors are called \textbf{rank-$1$} tensors. Every tensor is a finite linear combination of rank-$1$ tensors, and the smallest number of such terms in the combination is called the \textbf{rank} of the tensor.

\subsection{Tensor Networks}
There is an operation generalizing matrix multiplication to tensors: If $\mathcal{T} \in \mathbb{R}^{n_1 \times ... \times n_{i-1} \times n_i \times n_{i+1} ... \times n_d}$ and $\mathcal{S} \in \mathbb{R}^{m_1 \times ... \times m_{j-1} \times m_j \times m_{j+1} ... \times m_\ell}$ such that $n_i = m_j$, then their \textbf{contraction} along the $i^{\text{th}}$ and $j^{\text{th}}$ \textbf{modes} is the tensor $\mathcal{R} \in \mathbb{R}^{n_1 \times ... \times n_{i-1} \times n_{i+1} \times ... \times n_d \times m_1 \times ... \times m_{j-1} \times m_{j+1} \times ... \times m_\ell}$ with entries
$$\mathcal{R}_{t_1...t_{i-1}t_{i+1}...t_ds_1...s_{j-1}s_{j+1}...s_\ell} = \sum_{k = 1}^{n_i} \mathcal{T}_{t_1...t_{i-1}kt_{i+1}...t_d}\mathcal{S}_{s_1...s_{j-1}ks_{j+1}...s_\ell},$$
$$t_p \in [n_p], \; p \in [d], \; s_q \in [m_q], \; q \in [\ell]$$
Clearly, tensor contraction can be unwieldy. \textbf{Tensor network diagrams} are a graphical representation of tensors used to simplify contractions. A $d$-tensor is represented by a vertex with $d$ protruding edges. For example, two $3$-tensors $\mathcal{T} = \sum_{i = 1}^{r_\mathcal{T}} \mathbf{a}_i \otimes \mathbf{b}_i \otimes \mathbf{c}_i$ and $\mathcal{S} = \sum_{j = 1}^{r_\mathcal{S}} \mathbf{d}_j \otimes \mathbf{e}_j \otimes \mathbf{f}_j$ where $\mathbf{a}_i \in \mathbb{R}^{n_a}, \mathbf{b}_i \in \mathbb{R}^{n_b}, \mathbf{c}_i \in \mathbb{R}^{n_c}, \mathbf{d}_j \in \mathbb{R}^{n_d}, \mathbf{e}_j \in \mathbb{R}^{n_e}, \mathbf{f}_j \in \mathbb{R}^{n_f}$ are represented in Figures~\ref{len-2-train}a and \ref{len-2-train}b, where the dimension of each vector component labels an edge. If it is clear from the context, these edge labels can be dropped. Note that regardless of the number of tensor product terms in the sum defining $\mathcal{T}$ and $\mathcal{S}$, the same diagram of a vertex with three edges is used to represent them; what the diagram illustrates is the dimensionality of $\mathcal{T}$ and $\mathcal{S}$, as they are formed by tensor products of three vectors. If $n_c = n_f$, then $\mathcal{T}$ and $\mathcal{S}$ can be contracted along the $c$- and $f$-modes to form the tensor $\sum_{i = 1}^{r_\mathcal{T}}\sum_{j = 1}^{r_\mathcal{S}} \mathbf{a}_i \otimes \mathbf{b}_i \otimes \mathbf{d}_j \otimes \mathbf{e}_j \langle \mathbf{c}_i, \mathbf{f}_j \rangle$, shown in Figure~\ref{len-2-train}c.
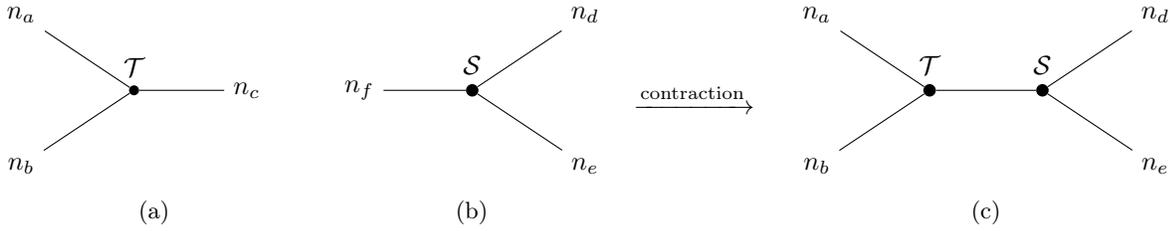
\begin{figure}[t]
    \begin{subfigure}{0.25\textwidth}
\begin{tikzpicture}[scale=1]
  \node (a1) at (-1.5,1) {$n_a$};  
  \node (a2) at (-1.5,-1)  {$n_b$};
  \node [style={circle,fill=black,scale=0.4,label={$\mathcal{T}$}}] (a3) at (0,0)  {};  
  \node (a4) at (1.5,0) {$n_c$};
  
  \draw (a1) -- (a3);
  \draw (a2) -- (a3);  
  \draw (a3) -- (a4);
\end{tikzpicture}
    \caption{}
    \end{subfigure}
    \begin{subfigure}{0.25\textwidth}
        \centering
\begin{tikzpicture}[scale=1]

  \node (a3) at (0,0)  {$n_f$};  
  \node [style={circle,fill=black,scale=0.5,label={$\mathcal{S}$}}] (a4) at (1.5,0) {};  
  \node (a5) at (3,1)  {$n_d$};  
  \node (a7) at (3,-1)  {$n_e$};

  \draw (a3) -- (a4);  
  \draw (a4) -- (a5);  
  \draw (a4) -- (a7);  
\end{tikzpicture}
    \caption{}
    \end{subfigure}
    \begin{subfigure}{0.0275\textwidth}
    \begin{equation*}
        \xrightarrow{\text{contraction}}
    \end{equation*}
    \begin{equation*}
        \small{}
    \end{equation*}
    \end{subfigure}
    \begin{subfigure}{0.5\textwidth}
        \centering
\begin{tikzpicture}[scale=1]
  \node (a1) at (-1.5,1) {$n_a$};  
  \node (a2) at (-1.5,-1)  {$n_b$};
  \node [style={circle,fill=black,scale=0.5,label={$\mathcal{T}$}}] (a3) at (0,0)  {};  
  \node [style={circle,fill=black,scale=0.5,label={$\mathcal{S}$}}] (a4) at (1.5,0) {};  
  \node (a5) at (3,1)  {$n_d$};  
  \node (a7) at (3,-1)  {$n_e$};

  \draw (a1) -- (a3);
  \draw (a2) -- (a3);  
  \draw (a3) -- (a4);  
  \draw (a4) -- (a5);  
  \draw (a4) -- (a7);  
\end{tikzpicture}
    \caption{}
    \end{subfigure}
    \caption{Carriages $\mathcal{T}$ and $\mathcal{S}$ contract to form a length $2$ tensor train}
    \label{len-2-train}
\end{figure}

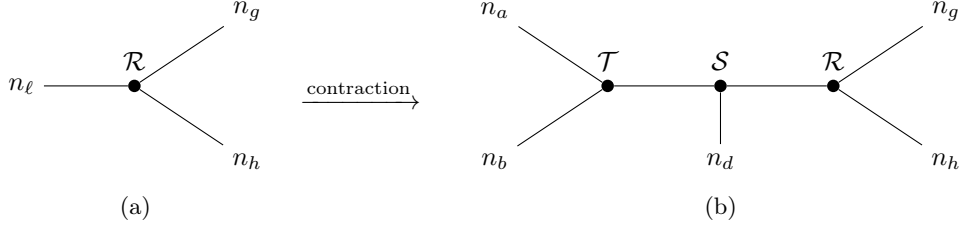
\begin{figure}[t]
\centering
    \begin{subfigure}{0.25\textwidth}
        \centering
\begin{tikzpicture}[scale=1]

  \node (a3) at (0,0)  {$n_\ell$};  
  \node [style={circle,fill=black,scale=0.5,label={$\mathcal{R}$}}] (a4) at (1.5,0) {};  
  \node (a5) at (3,1)  {$n_g$};  
  \node (a7) at (3,-1)  {$n_h$};

  \draw (a3) -- (a4);  
  \draw (a4) -- (a5);  
  \draw (a4) -- (a7);  
\end{tikzpicture}
    \caption{}
    \end{subfigure}
    \begin{subfigure}{0.08\textwidth}
    \begin{equation*}
        \xrightarrow{\text{contraction}}
    \end{equation*}
    \begin{equation*}
        \small{}
    \end{equation*}
    \end{subfigure}
    \begin{subfigure}{0.5\textwidth}
    \centering
\begin{tikzpicture}[scale=1]
  \node (a1) at (-1.5,1) {$n_a$};  
  \node (a2) at (-1.5,-1)  {$n_b$};
  \node [style={circle,fill=black,scale=0.5,label={$\mathcal{T}$}}] (a3) at (0,0)  {};  
  \node [style={circle,fill=black,scale=0.5,label={$\mathcal{S}$}}] (a4) at (1.5,0) {};  
  \node (a5) at (1.5,-1)  {$n_d$};  
  \node [style={circle,fill=black,scale=0.5,label={$\mathcal{R}$}}] (a6) at (3,0)  {};  
  \node (a7) at (4.5,1)  {$n_g$};
  \node (a8) at (4.5,-1)  {$n_h$};
  
  \draw (a1) -- (a3);
  \draw (a2) -- (a3);  
  \draw (a3) -- (a4);  
  \draw (a4) -- (a5);  
  \draw (a4) -- (a6);  
  \draw (a6) -- (a7);  
  \draw (a6) -- (a8);  
\end{tikzpicture}
\caption{}
\end{subfigure}
    \caption{Carriage $\mathcal{R}$ and length $2$ tensor train \ref{len-2-train}c contract to form a length $3$ tensor train}
    \label{len-3-train}
\end{figure}

If $\mathcal{R} = \sum_{k = 1}^{r_\mathcal{R}} \mathbf{g}_k \otimes \mathbf{h}_k \otimes \boldsymbol\ell_k$ where $\mathbf{g}_k \in \mathbb{R}^{n_g}, \mathbf{h}_k \in \mathbb{R}^{n_h}, \boldsymbol\ell_k \in \mathbb{R}^{n_\ell}$, represented by \ref{len-3-train}a in \textbf{Figure \ref{len-3-train}}, and $n_e = n_{\ell}$, then we can contract $\mathcal{R}$ with the tensor \ref{len-2-train}c along the $e$- and $\ell$-modes to form the tensor $\sum_{i = 1}^{r_\mathcal{T}}\sum_{j = 1}^{r_\mathcal{S}}\sum_{k = 1}^{r_\mathcal{R}} \mathbf{a}_i \otimes \mathbf{b}_i \otimes \mathbf{d}_j \otimes \mathbf{g}_k \otimes \mathbf{h}_k \langle \mathbf{c}_i, \mathbf{f}_j \rangle\langle \mathbf{e}_j, \boldsymbol\ell_k \rangle$, represented by \ref{len-3-train}b.

The tensor networks in Figures~\ref{len-2-train}c and \ref{len-3-train}b are called \textbf{tensor trains} of length $2$ and $3$, respectively, and generalize to longer trains. Tensor trains are \textbf{generated} by the sequential contraction of $3$-tensors, like $\mathcal{T}$, $\mathcal{S}$, and $\mathcal{R}$. These $3$-tensors at the vertices of the network will be referred to as \textbf{carriages}. {{Tensor trains can produce tensors of high rank:

\begin{lemma}\label{lem:rank} The rank of a length-$L$ tensor train with tensors $T_j = \sum_{i_j=1}^n \lambda^{(j)}_{i_j} (a^{(j)}_{i_j})^{\otimes 3}$ ($j=1,\ldots, L$) at its vertices, each of rank at most $n$, is at most $n^{L} =  n^{d-2}$, where $d$ is the order of the tensor.
\end{lemma}
\begin{proof}
Indeed, such a decomposition yields the following  tensor
$$\sum_{i_1,\ldots, i_L=1}^n\left[\lambda^{(1)}_{i_1}\cdots\lambda^{(L)}_{i_L} \langle a^{(1)}_{i_1}, a^{(2)}_{i_2}\rangle\cdots \langle a^{(L-1)}_{i_{L-1}}, a^{(L)}_{i_L}\rangle\right] (a^{(1)}_{i_1})^{\otimes 2}\otimes a^{(2)}_{i_2}\otimes\cdots\otimes a^{(L-1)}_{i_{L-1}}\otimes (a^{(L)}_{i_L})^{\otimes 2},$$
which  has rank at most $n^L$ and order $d$.
\end{proof}
}

Matrix product states often appear in the condensed matter physics literature with one additional vertex and edge attached to each end of the diagram, as opposed to our diagrams as in Figure \ref{len-2-train} \cite{quantum,TNNutshell}. This format can be converted to the format introduced here simply with a contraction of the single edges, thus removing the vertices at each end. Conversely, given a tensor network in the form of our diagram, a multiplication by the identity matrix at each end will yield the matrix product state form. Thus, as the two network formats are  equivalent, we proceed with the use of the term ``tensor train'' to describe the networks we study in this paper. The reason for the modified form is to better generalize and extend the class of tensors known as orthogonally decomposable tensors, as studied in \cite{kolda2015symmetric,Anandkumar2014,BDHR,Robeva,RobSei}.


}

\subsection{Orthogonally Decomposable Tensors and Tensor Trains}
A carriage such as $\mathcal{T}=\sum_{i=1}^{r_{\mathcal T}}\mathbf{a}_i\otimes\mathbf{b}_i\otimes\mathbf{c}_i$ is said to be \textbf{orthogonally decomposable} or \textbf{odeco} if $\{\mathbf{a}_i\}_{i = 1}^{r_{\mathcal{T}}} \subset \mathbb{R}^{n_a}$, $\{\mathbf{b}_i\}_{i = 1}^{r_{\mathcal{T}}} \subset \mathbb{R}^{n_b}$, $\{\mathbf{c}_i\}_{i = 1}^{r_{\mathcal{T}}} \subset \mathbb{R}^{n_c}$ are each a set of mutually orthogonal vectors. The carriage $\mathcal{T}$ is said to be \textbf{symmetric} if $\mathbf{a}_i = \mathbf{b}_i = \mathbf{c}_i$. In this case, $\mathcal{T} = \sum_{i = 1}^{r_\mathcal{T}} \mathbf{a}_i \otimes \mathbf{a}_i \otimes \mathbf{a}_i = \sum_{i = 1}^{r_\mathcal{T}} \mathbf{a}_i^{\otimes 3}$. Orthogonally decomposable tensors are appealing because they can be decomposed efficiently via the slice method~\cite{kolda2015symmetric} or via the tensor power method~\cite{Anandkumar2014}. For more on orthogonally decomposable tensors, please refer to~\cite{Anandkumar2014,BDHR,Robeva,RobSei}. A tensor train is said to be \textbf{symmetric/orthogonal} if it is generated by carriages that are symmetric/orthogonally decomposable. The goal of this paper is to decompose such tensors.

Notice that we can rewrite a carriage $\mathcal{T} = \sum_{i = 1}^{r_\mathcal{T}} \mathbf{a}_i \otimes \mathbf{b}_i \otimes \mathbf{c}_i = \sum_{i = 1}^{r_\mathcal{T}} \lambda_i\hat{\mathbf{a}}_i \otimes \hat{\mathbf{b}}_i \otimes \hat{\mathbf{c}}_i$, for some $\lambda_i \in \mathbb{R}$ and $\hat{\mathbf{a}}_i \subset \mathbb{R}^{n_a}$, $\hat{\mathbf{b}}_i \subset \mathbb{R}^{n_b}$, $\hat{\mathbf{c}}_i \subset \mathbb{R}^{n_c}$. For example, we could choose $\lambda_i = \|\mathbf{a}_i\|\|\mathbf{b}_i\|\|\mathbf{c}_i\|$ and $\hat{\mathbf{a}}_i = \frac{\mathbf{a}_i}{\|\mathbf{a}_i\|}$, $\hat{\mathbf{b}}_i = \frac{\mathbf{b}_i}{\|\mathbf{b}_i\|}$, $\hat{\mathbf{c}}_i = \frac{\mathbf{c}_i}{\|\mathbf{c}_i\|}$, in which case $\mathcal{T}$ is orthogonal if $\{\hat{\mathbf{a}}_i\}_{i = 1}^{r_{\mathcal{T}}} \subset \mathbb{R}^{n_a}$, $\{\hat{\mathbf{b}}_i\}_{i = 1}^{r_{\mathcal{T}}} \subset \mathbb{R}^{n_b}$, $\{\hat{\mathbf{c}}_i\}_{i = 1}^{r_{\mathcal{T}}} \subset \mathbb{R}^{n_c}$ are each an orthonormal set. $r_{\mathcal{T}}$ is said to be the \textbf{rank} of $\mathcal T$, $\hat{\mathbf{a}}_i, \hat{\mathbf{b}}_i, \hat{\mathbf{c}}_i$ the \textbf{vectors}, and $\lambda_i$ the \textbf{coefficients} of $\mathcal{T}$. To decompose a tensor train, we therefore mean to find the ranks, vectors, and coefficients of each carriage which generate it. Our goal is to decompose tensors which are known to have an orthogonal tensor train form; we are not aware of a procedure for testing whether a given tensor has an orthogonal tensor train decomposition. The methods we present in Sections \ref{sec:3}, \ref{sec:4}, and \ref{sec:5} for decomposing tensor trains can be used reliably when there is no noise in the tensor. In Section \ref{sec:7}, we perform numerical experiments to test our methods both in the absence and presence of noise. One of the main ambitions of tensor network research is to find efficient decompositions. Tensor network diagrams can have line segments, which correspond to tensor trains and is what we study in this paper, as well as loops, which correspond to \textbf{tensor rings} and have been studied in \cite{Jianfeng}. For more on tensor networks, please refer to~\cite{quantum,TNNutshell}.

Lastly, there are objects which we will encounter, such as scalars and orthonormal sets of vectors, which lie in a Zariski-closed subset of $\mathbb{R}^n$, for some $n$ pertaining to each object. We will say that such an object is \textbf{generic} if it lies in some Zariski-open subset of this closed set. In other words, if an object is randomly drawn from this closed set, then with probability 1, it will be drawn from the open subset. We also abbreviate ``up to permutation and sign'' as \textbf{UTPS}.

\section{Symmetric Orthogonal Decomposition of Tensor Trains of Length~2}\label{sec:3}

In this section we investigate the problem of decomposing an order-4 tensor $\mathcal T$ according to a tensor train network (cf. Figure~\ref{fig:TTLength2}) such that the carriages $\mathcal A$ and $\mathcal B$ are symmetric. We begin in Section~\ref{sec:3.1} by also assuming that $\mathcal A$ and $\mathcal B$ are symmetric orthogonally decomposable, and we show in Section~\ref{sec:3.2} how to use whitening in the non-orthogonal case.
\subsection{Symmetric Orthogonal Decomposition of Length 2}\label{sec:3.1}

\begin{problem}{\label{sym-len-2-prob}}
Let $\mathcal{T} \in \mathbb{R}^{n \times n \times n \times n}$ be a $4$-tensor admitting the following decomposition
\begin{align}\label{eq:TTlength2}
\mathcal{T} = \sum_{i=1}^{r_\mathcal{A}} \sum_{j=1}^{r_\mathcal{B}} \lambda_i \mu_j \mathbf{u}_i^{\otimes 2} \otimes \mathbf{v}_j^{\otimes 2} \langle \mathbf{u}_i, \mathbf{v}_j \rangle
\end{align}
where $\{\mathbf{u}_i\}_{i = 1}^{r_\mathcal{A}},\{\mathbf{v}_j\}_{j = 1}^{r_\mathcal{B}} \subset \mathbb{R}^n$ are generic orthonormal sets, and $\lambda_i,\mu_j \in \mathbb{R}\setminus\{0\}$. In other words, assume that $\mathcal{T}$ is a tensor train generated by
$$\mathcal{A} = \sum_{i=1}^{r_\mathcal{A}} \lambda_i \mathbf{u}_i^{\otimes 3} \hspace{1cm}\text{and}\hspace{1cm} \mathcal{B} = \sum_{j=1}^{r_\mathcal{B}} \mu_j \mathbf{v}_j^{\otimes 3}.$$

\begin{figure}
\begin{center}
\begin{tikzpicture}

  \node (a1) at (-1.5,1) {};  
  \node (a2) at (-1.5,-1)  {};
  \node [style={circle,fill=black,scale=0.5,label={$\mathcal{A}$}}] (a3) at (0,0)  {};  
  \node [style={circle,fill=white,scale=0.01}] (a4) at (1.5,0) {};  
  \node [style={circle,fill=black,scale=0.5,label={$\mathcal{B}$}}] (a5) at (3,0)  {};  
  \node (a6) at (4.5,1)  {};
  \node (a7) at (4.5,-1)  {};
  
  \draw (a1) -- (a3);
  \draw (a2) -- (a3);  
  \draw (a3) -- (a4);  
  \draw (a4) -- (a5);  
  \draw (a5) -- (a6);  
  \draw (a5) -- (a7);  

\end{tikzpicture}
\end{center}
\caption{A tensor network diagram of $\mathcal{T}$}
\label{fig:TTLength2}
\end{figure}
Given $\mathcal T$ as in~\eqref{eq:TTlength2}, find the above decomposition, including all ranks, vectors, and coefficients.
\end{problem}



\medskip

We solve this problem by adapting Kolda's slice method \cite{kolda2015symmetric}. Consider generic weighted sums over the last two indices and first two indices of $\mathcal{T}$
$$\mathbf{S}_\mathcal{A} = \sum_{i_3 = 1}^n \sum_{i_4 = 1}^n \alpha_{i_3 i_4} \mathcal{T}(:, :, i_3, i_4) = \sum_{i=1}^{r_\mathcal{A}} \sigma_i \mathbf{u}_i \mathbf{u}_i^\top = \mathbf{U}\boldsymbol\Sigma\mathbf{U}^\top,$$
$$\mathbf{S}_\mathcal{B} = \sum_{i_1 = 1}^n \sum_{i_2 = 1}^n \beta_{i_1 i_2} \mathcal{T}(i_1, i_2, :, :) = \sum_{j=1}^{r_\mathcal{B}} \gamma_j \mathbf{v}_j \mathbf{v}_j^\top = \mathbf{V}\boldsymbol\Gamma\mathbf{V}^\top,$$
for some generic $\alpha_{i_3 i_4},\beta_{i_1 i_2},\sigma_i,\gamma_j \in \mathbb{R}$. Here $\boldsymbol\Sigma,\boldsymbol\Gamma \in \mathbb{R}^{n \times n}$ are diagonal with $\sigma_i$ or $0$ and $\gamma_j$ or $0$ along their diagonals, respectively, and  $\mathbf{U},\mathbf{V} \in \mathbb{R}^{n \times n}$ are orthogonal matrices with $\mathbf{u}_i$ and $\mathbf{v}_j$ in $r_\mathcal{A}$ and $r_\mathcal{B}$ of their columns, respectively. The vectors $\mathbf{u}_i$ and $\mathbf{v}_j$ can, therefore, be found UTPS via an eigendecomposition on $\mathbf{S}_\mathcal{A}$ and $\mathbf{S}_\mathcal{B}$, since they correspond to the nonzero eigenvectors. Thus, we have found all ranks and vectors. To find the coefficients, construct the matrix $\mathbf{R} \in \mathbb{R}^{r_{\mathcal{A}}\times r_{\mathcal{B}}}$ with entries
$$\mathbf{R}_{\hat{i}\hat{j}} = \frac{\mathcal{T}(\mathbf{u}_{\hat{i}},\mathbf{v}_{\hat{j}})}{\langle \mathbf{u}_{\hat{i}}, \mathbf{v}_{\hat{j}} \rangle} = \lambda_{\hat{i}}\mu_{\hat{j}}.$$
Note that $\langle \mathbf{u}_{\hat{i}}, \mathbf{v}_{\hat{j}} \rangle \neq 0$ since $\mathbf{u}_i$ and $\mathbf{v}_j$ are generic. Then, $\mathbf{R} = \boldsymbol\lambda\boldsymbol\mu^\top$ is a rank-$1$ matrix, where $\boldsymbol\lambda$ and $\boldsymbol\mu$ are the vectors whose components are $\lambda_i$ and $\mu_j$. Performing Singular Value Decomposition (SVD) on $\mathbf{R}$, if $\boldsymbol\zeta$ and $\boldsymbol\eta$ are the first left and right singular vectors and $\tau$ is the largest non-zero singular value, then $\boldsymbol\lambda = \tau\boldsymbol\zeta$ and $\boldsymbol\mu = \boldsymbol\eta$ is a solution. This non-uniqueness is due to non-zero scaling.

Pseudocode for this procedure can be found in Algorithm~\ref{symmL2_alg}.

\subsection{Whitening: Symmetric Non-Orthogonal Decomposition of Length  2}\label{sec:3.2}

If $\mathcal{T}$ has a decomposition as in~\eqref{eq:TTlength2} where $\{\mathbf{u}_i\}_{i = 1}^{r_{\mathcal{L}}}$ and $\{\mathbf{v}_j\}_{j = 1}^{r_{\mathcal{R}}}$ are linearly independent but not orthonormal sets, then we can adapt Kolda's method of whitening \cite{kolda2015symmetric}. The generic weighted sums over indices of $\mathcal{T}$ have two different decompositions: one in terms of the vectors $\mathbf{u}_i$ and $\mathbf{v}_j$, and one in terms of its ``skinny'' eigendecomposition:
$$\mathbf{C}_\mathcal{A} = \sum_{i_3 = 1}^n \sum_{i_4 = 1}^n \rho_{i_3 i_4} \mathcal{T}(:, :, i_3, i_4) = \sum_{i = 1}^{r_\mathcal{A}} \widetilde{\sigma}_i \mathbf{u}_i \mathbf{u}_i^\top = \widetilde{\mathbf{U}} \widetilde{\boldsymbol{\Sigma}} \widetilde{\mathbf{U}}^\top = \mathbf{X}_{\mathcal{A}}\mathbf{D}_{\mathcal{A}}\mathbf{X}_{\mathcal{A}}^\top$$
$$\mathbf{C}_\mathcal{B} = \sum_{i_1 = 1}^n \sum_{i_2 = 1}^n \tau_{i_1 i_2} \mathcal{T}(i_1, i_2, :, :) = \sum_{j = 1}^{r_\mathcal{B}} \widetilde{\gamma}_j \mathbf{v}_j \mathbf{v}_j^\top = \widetilde{\mathbf{V}} \widetilde{\boldsymbol{\Gamma}} \widetilde{\mathbf{V}}^\top = \mathbf{X}_{\mathcal{B}}\mathbf{D}_{\mathcal{B}}\mathbf{X}_{\mathcal{B}}^\top$$
where $\widetilde{\boldsymbol\Sigma}, \mathbf{D}_{\mathcal{A}} \in \mathbb{R}^{r_{\mathcal{A}} \times r_{\mathcal{A}}}, \widetilde{\boldsymbol\Gamma}, \mathbf{D}_{\mathcal{B}} \in \mathbb{R}^{r_{\mathcal{B}} \times r_{\mathcal{B}}}$ are diagonal with non-zero diagonal entries, $\widetilde{\mathbf{U}} \in \mathbb{R}^{n \times r_\mathcal{A}}$ and $\widetilde{\mathbf{V}} \in \mathbb{R}^{n \times r_\mathcal{B}}$ have $\mathbf{u}_i$ and $\mathbf{v}_j$ as columns, and $\mathbf{X}_{\mathcal{A}} \in \mathbb{R}^{n \times r_{\mathcal{A}}}$ and $\mathbf{X}_{\mathcal{B}} \in \mathbb{R}^{n \times r_{\mathcal{B}}}$ are orthogonal matrices. If $\mathbf{C}_\mathcal{A}$ and $\mathbf{C}_\mathcal{B}$ are positive semi-definite (PSD), then $\widetilde{\boldsymbol{\Sigma}}$, $\mathbf{D}_{\mathcal{A}}$, $\widetilde{\boldsymbol{\Gamma}}$, and $\mathbf{D}_{\mathcal{B}}$ have positive diagonal entries. Hence, let
$$\mathbf{W}_\mathcal{A} = \mathbf{D}_{\mathcal{A}}^{-\frac{1}{2}} \mathbf{X}_{\mathcal{A}}^\top \in \mathbb{R}^{r_\mathcal{A} \times n} \hspace{2cm} \mathbf{W}_\mathcal{B} = \mathbf{D}_{\mathcal{B}}^{-\frac{1}{2}} \mathbf{X}_{\mathcal{B}}^\top \in \mathbb{R}^{r_\mathcal{B} \times n}$$
Then we see, for instance,
$$\left(\mathbf{W}_\mathcal{A} \widetilde{\mathbf{U}} \widetilde{\boldsymbol{\Sigma}}^{-\frac{1}{2}}\right)\left(\mathbf{W}_\mathcal{A} \widetilde{\mathbf{U}} \widetilde{\boldsymbol{\Sigma}}^{-\frac{1}{2}}\right)^\top = \mathbf{W}_\mathcal{A} \left(\widetilde{\mathbf{U}} \widetilde{\boldsymbol{\Sigma}} \widetilde{\mathbf{U}}^\top\right) \mathbf{W}_\mathcal{A}^\top = \mathbf{D}_{\mathcal{A}}^{-\frac{1}{2}}\mathbf{X}_{\mathcal{A}}^\top\left(\mathbf{X}_{\mathcal{A}}\mathbf{D}_{\mathcal{A}}\mathbf{X}_{\mathcal{A}}^\top\right)\mathbf{X}_{\mathcal{A}}\mathbf{D}_{\mathcal{A}}^{-\frac{1}{2}} = \mathbf{I}_{r_\mathcal{A} \times r_\mathcal{A}}$$
which shows that $\mathbf{W}_\mathcal{A} \widetilde{\mathbf{U}} \widetilde{\boldsymbol{\Sigma}}^{-\frac{1}{2}} \in \mathbb{R}^{r_\mathcal{A} \times r_\mathcal{A}}$ is orthogonal. Thus, $\overline{\mathbf{U}} = \mathbf{W}_\mathcal{A} \widetilde{\mathbf{U}} \in \mathbb{R}^{r_\mathcal{A} \times r_\mathcal{A}}$ has orthogonal columns, and similarly for $\overline{\mathbf{V}} = \mathbf{W}_\mathcal{B} \widetilde{\mathbf{V}} \in \mathbb{R}^{r_\mathcal{B} \times r_\mathcal{B}}$. This means that the tensor train obtained by contracting $\mathcal T$ with the matrix $\mathbf{W}_{\mathcal A}$ along its left two dangling edges and with the matrix $\mathbf{W}_{\mathcal B}$ along its right two dangling edges:
$$\overline{\mathcal{T}} = \mathcal{T}(\mathbf{W}_\mathcal{A}, \mathbf{W}_\mathcal{A}, \mathbf{W}_\mathcal{B}, \mathbf{W}_\mathcal{B}) = \sum_{i=1}^{r_\mathcal{A}} \sum_{j=1}^{r_\mathcal{B}} \lambda_i \mu_j (\mathbf{W}_\mathcal{A} \mathbf{u}_i)^{\otimes 2} \otimes (\mathbf{W}_\mathcal{B} \mathbf{v}_j)^{\otimes 2} \langle \mathbf{u}_i, \mathbf{v}_j \rangle$$
$$= \sum_{i=1}^{r_\mathcal{A}} \sum_{j=1}^{r_\mathcal{B}} \widetilde{\lambda}_i \widetilde{\mu}_j \left(\frac{\mathbf{W}_\mathcal{A} \mathbf{u}_i}{\|\mathbf{W}_\mathcal{A} \mathbf{u}_i\|}\right)^{\otimes 2} \otimes \left(\frac{\mathbf{W}_\mathcal{B} \mathbf{v}_j}{\|\mathbf{W}_\mathcal{B} \mathbf{v}_j\|}\right)^{\otimes 2} \langle \mathbf{u}_i, \mathbf{v}_j \rangle \in \mathbb{R}^{r_{\mathcal{A}} \times r_{\mathcal{A}} \times r_{\mathcal{B}} \times r_{\mathcal{B}}}$$
for appropriate $\widetilde{\lambda}_i,\widetilde{\mu}_j \in \mathbb{R}$, is symmetric and orthogonal and can be decomposed by the method described earlier. If $\overline{\mathbf{u}}_i$ are the vectors found from the eigendecomposition, then one can recover
$$\mathbf{u}_i = \frac{\mathbf{W}_\mathcal{A}^{\dagger} \overline{\mathbf{u}}_i}{\|\mathbf{W}_\mathcal{A}^{\dagger} \overline{\mathbf{u}}_i\|} = \frac{\mathbf{X}_{\mathcal{A}}\mathbf{D}_{\mathcal{A}}^{\frac{1}{2}} \overline{\mathbf{u}}_i}{\|\mathbf{X}_{\mathcal{A}}\mathbf{D}_{\mathcal{A}}^{\frac{1}{2}} \overline{\mathbf{u}}_i\|}$$
and similarly for $\mathbf{v}_j$. Algorithm~\ref{symmL2_alg} provides pseudocode for decomposing symmetric tensor trains of length $2$ in either the orthogonal or non-orthogonal case{{, and its correctness is summarized by the following result:

\begin{theorem}
If $\mathcal{T} \in \mathbb{R}^{n\times n \times n \times n}$ is a $4$-tensor generated by carriages $\displaystyle \mathcal{A} = \sum_{i = 1}^{r_{\mathcal{A}}} \lambda_{i}\mathbf{u}_{i}^{\otimes 3}$ and $\displaystyle \mathcal{B} = \sum_{j = 1}^{r_{\mathcal{B}}} \mu_{j}\mathbf{v}_{j}^{\otimes 3}$ where $\lambda_{i},\mu_{j} \in \mathbb{R}$ are generic coefficients and $\{\mathbf{u}_{i}\}_{i = 1}^{r_{\mathcal{A}}}, \{\mathbf{v}_{j}\}_{j = 1}^{r_{\mathcal{B}}} \subset \mathbb{R}^n$ are generic vectors, then Algorithm \ref{symmL2_alg} recovers $r_{\mathcal{A}}$ and $r_{\mathcal{B}}$, and $\lambda_i$, $\mu_j$, $\{\mathbf{u}_{i}\}_{i = 1}^{r_{\mathcal{A}}}$ and $\{\mathbf{v}_{j}\}_{j = 1}^{r_{\mathcal{B}}}$ UTPS.
\end{theorem}

}}

\begin{algorithm}
\caption{Symmetric Decomposition of Tensor Trains of Length 2}\label{symmL2_alg}

\textbf{Input:}

$\mathcal{T} \in \mathbb{R}^{n \times n \times n \times n}$ that decomposes according to a symmetric tensor train of length 2. \\An indicator if $\mathcal A$ and $\mathcal B$ are non-orthogonal, and therefore we need to apply whitening first.

\textbf{Output:}

$\boldsymbol{\lambda} \in \mathbb{R}^{r_{\mathcal{A}}}$, $\{\mathbf{u}_i\}_{i = 1}^{r_{\mathcal{A}}} \subset \mathbb{R}^n$, the coefficients and (orthonormal) vectors forming $\mathcal{A}$

$\boldsymbol{\mu} \in \mathbb{R}^{r_{\mathcal{B}}}$, , $\{\mathbf{v}_j\}_{j = 1}^{r_{\mathcal{B}}} \subset \mathbb{R}^n$, similarly for $\mathcal{B}$

\begin{algorithmic}[1]
\If {whitening}
\Repeat
    \State $\rho, \tau \gets $ generic real $n \times n$ matrices
    \State $\mathbf{C}_\mathcal{A} \gets \sum_{i_3 = 1}^n \sum_{i_4 = 1}^n \rho_{i_3 i_4} \mathcal{T}(:, :, i_3, i_4)$
    \State $\mathbf{C}_\mathcal{B} \gets \sum_{i_1 = 1}^n \sum_{i_2 = 1}^n \tau_{i_1 i_2} \mathcal{T}(i_1, i_2, :, :)$
\Until{$\mathbf{C}_\mathcal{A}, \mathbf{C}_\mathcal{B}$ are PSD or exit with failure}
\State $\mathbf{X}_\mathcal{A} \mathbf{D}_\mathcal{A} \mathbf{X}_\mathcal{A}^\top \gets $ ``skinny" eigendecomposition of $\mathbf{C}_\mathcal{A}$
\State $\mathbf{X}_\mathcal{B} \mathbf{D}_\mathcal{B} \mathbf{X}_\mathcal{B}^\top \gets $ ``skinny" eigendecomposition of $\mathbf{C}_\mathcal{B}$
\State $\mathbf{W}_\mathcal{A} \gets \mathbf{D}_\mathcal{A}^{-\frac{1}{2}} \mathbf{X}_\mathcal{A}^\top$
\State $\mathbf{W}_\mathcal{B} \gets \mathbf{D}_\mathcal{B}^{-\frac{1}{2}} \mathbf{X}_\mathcal{B}^\top$
\State $\overline{\mathcal{T}} \gets \mathcal{T}(\mathbf{W}_\mathcal{A}, \mathbf{W}_\mathcal{A}, \mathbf{W}_\mathcal{B}, \mathbf{W}_\mathcal{B})$
\Else
\State $\overline{\mathcal{T}} \gets \mathcal{T}$
\EndIf
\State $\alpha, \beta \gets $ generic real $n \times n$ matrices
\State $\mathbf{S}_\mathcal{A} \gets \sum_{i_3 = 1}^n \sum_{i_4 = 1}^n \alpha_{i_3 i_4} \overline{\mathcal{T}}(:, :, i_3, i_4)$
\State $\mathbf{S}_\mathcal{B} \gets \sum_{i_1 = 1}^n \sum_{i_2 = 1}^n \beta_{i_1 i_2} \overline{\mathcal{T}}(i_1, i_2, :, :)$
\State $\{\sigma_i, \mathbf{\overline{u}}_i\}_{i=1}^{r_\mathcal{A}} \gets $ eigenpairs of $\mathbf{S}_\mathcal{A}$ with $\sigma_i \neq 0$
\State $\{\gamma_j, \mathbf{\overline{v}}_j\}_{j=1}^{r_\mathcal{B}} \gets $ eigenpairs of $\mathbf{S}_\mathcal{B}$ with $\gamma_j \neq 0$
    \If {whitening}
    \State $\mathbf{u}_i \gets \mathbf{W}_\mathcal{A}^\dagger \overline{\mathbf{u}}_i/\|\mathbf{W}_\mathcal{A}^\dagger \overline{\mathbf{u}}_i\|, \; i \in [r_{\mathcal{A}}]$
    \Else
    \State $\mathbf{u}_i \gets \overline{\mathbf{u}}_i, \; i \in [r_{\mathcal{A}}]$
    \EndIf
    \If {whitening}
    \State $\mathbf{v}_j \gets \mathbf{W}_\mathcal{B}^\dagger \overline{\mathbf{v}}_j/\|\mathbf{W}_\mathcal{B}^\dagger \overline{\mathbf{v}}_j\|, \; j \in [r_{\mathcal{B}}]$
    \Else
    \State $\mathbf{v}_j \gets \overline{\mathbf{v}}_j, \; j \in [r_{\mathcal{B}}]$
    \EndIf
    \State $\boldsymbol{\mathbf{R}}_{ij} \gets \mathcal{T}(\overline{\mathbf{u}}_i,\overline{\mathbf{u}}_i,\overline{\mathbf{v}}_j,\overline{\mathbf{v}}_j) / \langle \mathbf{u}_i, \mathbf{v}_j \rangle, \; i \in [r_{\mathcal{A}}], \; j \in [r_{\mathcal{B}}]$
\State $\boldsymbol{\lambda} \gets $ largest singular value multiplied by corresponding left singular vector from SVD of $\mathbf{R}$
\State $\boldsymbol{\mu} \gets $ corresponding right singular vector from SVD of $\mathbf{R}$
\end{algorithmic}
\end{algorithm}

\subsection{Remarks}{\label{symm-odeco-2-remark}}
It is not guaranteed that $\mathbf{C}_\mathcal{A}$ and $\mathbf{C}_\mathcal{B}$ are PSD, and hence if they are not, we recompute generic weighted sums over indices of $\mathcal{T}$ until they are PSD or exit with failure. We analyze the performance of Algorithm~\ref{symmL2_alg} in Section~\ref{sec:7}. It is also straightforward to generalize to the case where each carriage in the train is a symmetric $m$-tensor, for some $m > 3$. Furthermore, if there are $p$ contracted edges between $\mathcal{A}$ and $\mathcal{B}$, then one simply divides by $\langle \mathbf{u}_i, \mathbf{v}_j \rangle ^p$ when forming the matrix $\mathbf{R}$ from Section~\ref{sec:3.1}.

\subsection{Applications}\label{sec:3.4}

The tensor decomposition discussed in Section~\ref{sec:3.2} is equivalent to the decomposition of the joint distribution of the four leaf variables in the graph in Figure~\ref{fig:GraphicalModel} below (note that the joint distribution of four variables each taking $n$ values is precisely an $n\times n\times n\times n$ tensor). This follows directly from the duality between tensor networks and graphical models in~\cite{RobSei2}. Note that the three hidden variables take $r$, $s$, and $n$ values respectively, and finding the decomposition discussed in Section~\ref{sec:3.2} recovers the joint distribution of {\em all} the variables in the graphical below.
\begin{figure}[H]
\begin{center}
\includegraphics[width=0.35\textwidth]{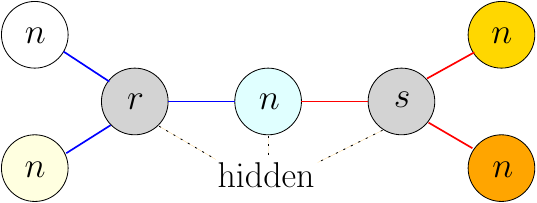}
\end{center}
\caption{A hidden variable graphical model equivalent (in the sense that they yield the same tensor decomposition) to a length-2 tensor train whose carriages are symmetric and have ranks $r$ and $s$, respectively.}\label{fig:GraphicalModel}
\end{figure}

\medskip

\section{Symmetric Orthogonal Decomposition of Tensor Trains of Length $L \geq 3$}\label{sec:4}

\medskip

In this section we consider longer tensor trains such that the carriages are symmetric and orthogonally decomposable.

\begin{center}
\begin{tikzpicture}{\label{len-3-train-fig}}

  \node (a1) at (-1.5,1) {};  
  \node (a2) at (-1.5,-1)  {};
  \node [style={circle,fill=black,scale=0.5,label={$\mathcal{X}_1$}}] (a3) at (0,0)  {};  
  \node [style={circle,fill=black,scale=0.5,label={$\mathcal{X}_2$}}] (a4) at (1.5,0) {};  
  \node (a5) at (1.5,-1)  {};  
  \node (a6) at (3,0)  {$\cdots$};  
  \node (a7) [style={circle,fill=black,scale=0.5,label={$\mathcal{X}_{L-1}$}}] at (4.5,0)  {};
  \node (a11) at (4.5,-1)  {};  
  \node [style={circle,fill=black,scale=0.5,label={$\mathcal{X}_L$}}] (a8) at (6,0)  {};
  \node (a9) at (7.5,1)  {};
  \node (a10) at (7.5,-1)  {};
  
  \draw (a1) -- (a3);
  \draw (a2) -- (a3);  
  \draw (a3) -- (a4);  
  \draw (a4) -- (a5);  
  \draw (a4) -- (a6);  
  \draw (a6) -- (a7);  
  \draw (a7) -- (a8);
  \draw (a8) -- (a9);
  \draw (a8) -- (a10);
  \draw (a7) -- (a11);
  
\node [below=1.25cm,align=flush center,text width=10cm] at (a6)
        {
        Figure 3: A tensor network diagram of $\mathcal{T}$
        };

\end{tikzpicture}
\end{center}

\begin{problem}
Let $L \geq 3$ and $\mathcal{T} \in \mathbb{R}^{n^{L+2}}$ be an $(L+2)$-tensor admitting the following decomposition
\begin{align}\label{eq:2}
\mathcal{T} = \sum_{i_1 = 1}^{r_{\mathcal{X}_1}}...\sum_{i_L = 1}^{r_{\mathcal{X}_L}} \lambda_{i_1}^1 ... \lambda_{i_L}^L (\mathbf{x}_{i_1}^1)^{\otimes 2} \otimes \mathbf{x}_{i_2}^2 \otimes ... \otimes \mathbf{x}_{i_{L-1}}^{L-1} \otimes (\mathbf{x}_{i_L}^L)^{\otimes 2} \langle \mathbf{x}_{i_1}^1, \mathbf{x}_{i_2}^2 \rangle ... \langle \mathbf{x}_{i_{L-1}}^{L-1}, \mathbf{x}_{i_L}^L \rangle
\end{align}
where for all $j \in [L]$, $\{\mathbf{x}_{i_j}^j\}_{i_j = 1}^{r_{\mathcal{X}_j}} \subset \mathbb{R}^n$ is a generic orthonormal set, and $\lambda_{i_j}^j \in \mathbb{R}$ are generic. In other words, assume that $\mathcal{T}$ (cf. Figure \ref{len-3-train-fig}) is a tensor train generated by
$$\mathcal{X}_j = \sum_{i_j = 1}^{r_{\mathcal{X}_j}} \lambda_{i_j}^j (\mathbf{x}_{i_j}^j)^{\otimes 3}$$

Given $\mathcal T$, find the decomposition~\eqref{eq:2}, including all ranks, vectors, and coefficients.
\end{problem}

The assumption that $\{\mathbf{x}_{i_j}^j\}_{i_j = 1}^{r_{\mathcal{X}_j}}$ and $\lambda_{i_j}^j$ are generic  will be important for our results on how to decompose $\mathcal{T}$, as the results may not hold otherwise. Because of the genericity assumption, our results hold with probability $1$. We present a solution when $\mathcal{T}$ satisfies the following condition:
{{

\begin{definition}{\label{drc}}
The tensor train $\mathcal{T}$ satisfies the \textbf{Decreasing Ranks Condition (DRC)} if there exists a $j \in [L]$ such that $r_{\mathcal{X}_1} \geq r_{\mathcal{X}_2} \geq ... \geq r_{\mathcal{X}_{j-1}} \geq r_{\mathcal{X}_j}$ and $r_{\mathcal{X}_j} \leq r_{\mathcal{X}_{j+1}} \leq ... \leq r_{\mathcal{X}_{L-1}} \leq r_{\mathcal{X}_L}$.
\end{definition}}

}

The DRC essentially states that there is not too much information lost about the ranks, vectors, and coefficients of the carriages upon contraction. The steps of the solution involve a method called \textbf{kernel completion}, which we describe in detail later, and are summarized below. We present pseudocode in Algorithm~\ref{symm-odeco-alg}.
\begin{enumerate}
    \item Find $\{\mathbf{x}_{i_1}^1\}_{i_1 = 1}^{r_{\mathcal{X}_1}}$ via an eigendecomposition of $\mathcal{T}(\cdot,\cdot,\mathbf{v},... ,\mathbf{v})$, where $\mathbf{v} \in \mathbb{R}^n$ is a generic vector. Then sequentially apply kernel completion starting from $\mathcal{X}_2$ and ending at $\mathcal{X}_{L-1}$, obtaining a collection of orthonormal vectors $\{\mathbf{x}_{i_2}^{2,\text{LR}}\}_{i_2 = 1}^{r_{\mathcal{X}_2,\text{LR}}}$, $...$, $\{\mathbf{x}_{i_{L-1}}^{L-1,\text{LR}}\}_{i_{L-1} = 1}^{r_{\mathcal{X}_{L-1},\text{LR}}}$.
    \item Find $\{\mathbf{x}_{i_L}^L\}_{i_L = 1}^{r_{\mathcal{X}_L}}$ via an eigendecomposition of $\mathcal{T}(\mathbf{v},...,\mathbf{v},\cdot,\cdot)$, where $\mathbf{v} \in \mathbb{R}^n$ is a generic vector. Then sequentially apply kernel completion starting from $\mathcal{X}_{L-1}$ and ending at $\mathcal{X}_2$, obtaining a collection of orthonormal vectors $\{\mathbf{x}_{i_{L-1}}^{L-1,\text{RL}}\}_{i_{L-1} = 1}^{r_{\mathcal{X}_{L-1},\text{RL}}}$, $...$, $\{\mathbf{x}_{i_2}^{2,\text{RL}}\}_{i_2 = 1}^{r_{\mathcal{X}_2,\text{RL}}}$.
    \item For each $2 \leq j \leq L-1$, choose an orthonormal set $\{\mathbf{x}_{i_j}^j\}_{i_j = 1}^{r_{\mathcal{X}_j}}$ from $\{\mathbf{x}_{i_j}^{j,\text{LR}}\}_{i_j = 1}^{r_{\mathcal{X}_j,\text{LR}}}$ and $\{\mathbf{x}_{i_j}^{j,\text{RL}}\}_{i_j = 1}^{r_{\mathcal{X}_j,\text{RL}}}$ based on whichever of $r_{\mathcal{X}_j,\text{LR}}$ or $r_{\mathcal{X}_j,\text{RL}}$ is greater. If they are equal, choose either set.
    \item Construct an $L$-tensor $\mathcal{R} \in \mathbb{R}^{r_{\mathcal{X}_1}\times...\times r_{\mathcal{X}_L}}$ whose $\hat{i}_1...\hat{i}_L$-entry is given by
    \begin{align}\label{eq:Rtensor}\mathcal{R}_{\hat{i}_1...\hat{i}_L} = \frac{\mathcal{T}(\mathbf{x}_{\hat{i}_1}^1,...,\mathbf{x}_{\hat{i}_L}^L)}{\prod_{k = 1}^{L-1}\langle \mathbf{x}_{\hat{i}_k}^k, \mathbf{x}_{\hat{i}_{k+1}}^{k+1} \rangle} = \lambda_{\hat{i}_1}^1...\lambda_{\hat{i}_L}^L
    \end{align}
    for all $\hat{i}_j \in [r_{\mathcal{X}_j}]$ and $j \in [L]$. Then $\mathcal{R} = \boldsymbol\lambda_1\otimes...\otimes\boldsymbol\lambda_L$, where $\boldsymbol\lambda_j \in \mathbb{R}^{r_{\mathcal{X}_j}}$ is the vector whose entries are $\lambda_{i_j}^j$. Apply the rank-$1$ alternating least squares~\cite{KoBa09} algorithm on $\mathcal{R}$ to obtain all the coefficients.
\end{enumerate}
We now explain each of these steps in more detail.

\textbf{Step 1: Decomposing the train from left to right}

We can always find the vectors at the ``ends'' of the train $\{\mathbf{x}_{i_1}^1\}$ and $\{\mathbf{x}_{i_L}^L\}$ UTPS regardless of whether or not $\mathcal{T}$ satisfies the DRC. Let $\mathbf{v} \in \mathbb{R}^n$ be a generic vector and consider $\overline{\mathbf{X}_1} = \mathcal{T}(\cdot,\cdot,\mathbf{v},...,\mathbf{v})$
$$\overline{\mathbf{X}_1} = \sum_{i_1 = 1}^{r_{\mathcal{X}_1}}\left(\sum_{i_2 = 1}^{r_{\mathcal{X}_2}}...\sum_{i_L = 1}^{r_{\mathcal{X}_L}} \lambda_{i_1}^1 ... \lambda_{i_L}^L\langle \mathbf{x}_{i_1}^1, \mathbf{x}_{i_2}^2 \rangle ... \langle \mathbf{x}_{i_{L-1}}^{L-1}, \mathbf{x}_{i_L}^L \rangle\langle \mathbf{x}_{i_2}^2, \mathbf{v} \rangle ... \langle \mathbf{x}_{i_{L-1}}^{L-1}, \mathbf{v} \rangle\langle \mathbf{x}_{i_L}^L, \mathbf{v} \rangle^2\right) (\mathbf{x}_{i_1}^1)^{\otimes 2}$$
$$= \sum_{i_1 = 1}^{r_{\mathcal{X}_1}} \sigma_{i_1} \mathbf{x}_{i_1}^1(\mathbf{x}_{i_1}^1)^\top = \widetilde{\mathbf{X}_1}\boldsymbol{\Sigma} \widetilde{\mathbf{X}_1}^\top$$
for some $\sigma_i \in \mathbb{R}$, diagonal $\boldsymbol{\Sigma} \in \mathbb{R}^{n \times n}$ with $\sigma_i$ or $0$ along the diagonal, and orthogonal $\widetilde{\mathbf{X}_1} \in \mathbb{R}^{n \times n}$ with $\{\mathbf{x}_{i_1}^1\}$ in $r_{\mathcal{X}_1}$ of its columns. The set of vectors $\{\mathbf{x}_{i_1}^1\}$ can therefore be found UTPS via an eigendecomposition of $\overline{\mathbf{X}}_1$, as they correspond to the non-zero eigenvalues. The same method can be used to find $\widetilde{\mathbf{X}_L} \in \mathbb{R}^{n \times n}$, which has $\{\mathbf{x}_{i_L}^L\}$ in $r_{\mathcal{X}_L}$ of its columns, UTPS. Note that $\{\mathbf{x}_{i_1}^1\}$ and $\{\mathbf{x}_{i_L}^L\}$ can also be found using Kolda's method as with trains of length 2, which forms an alternative $\overline{\mathbf{X}}_1$ by constructing generic weighted sums over all but two indices of $\mathcal{T}$
$$\overline{\mathbf{X}_1} = \sum_{i_3 = 1}^n ... \sum_{i_{L+2} = 1}^n \gamma_{i_3,...,i_{L+2}}\mathcal{T}(:,:,i_3,...,i_{L+2}),$$
where $\gamma_{i_3,...,i_{L+2}} \in \mathbb{R}$ are generic. We found our method to be efficient and simple to implement for long trains, hence we use it.

{\textbf{Kernel completion}}

Complete the set of orthonormal vectors $\{\mathbf{x}_{i_1}^1\}_{i_1 = 1}^{r_{\mathcal{X}_1}}$ found above to an orthonormal basis and let $\mathbf{X}_1 \in \mathbb{R}^{n \times n}$ be the orthogonal matrix whose columns are these basis vectors. Denote by $\mathbf{X}_2 \in \mathbb{R}^{n \times n}$ an orthogonal matrix whose first $r_{\mathcal{X}_2}$ columns are $\{\mathbf{x}_{i_2}^2\}$, let $\mathbf{v} \in \mathbb{R}^n$ be a generic vector, and consider $\mathbf{T} = \mathcal{T}(\mathbf{v},\cdot,\cdot,\mathbf{v},...,\mathbf{v})$
$$\mathbf{T} = \sum_{i_1 = 1}^{r_{\mathcal{X}_1}}...\sum_{i_L = 1}^{r_{\mathcal{X}_L}} \lambda_{i_1}^1 ... \lambda_{i_L}^L \mathbf{x}_{i_1}^1(\mathbf{x}_{i_2}^2)^\top\langle \mathbf{x}_{i_1}^1, \mathbf{x}_{i_2}^2 \rangle ... \langle \mathbf{x}_{i_{L-1}}^{L-1}, \mathbf{x}_{i_L}^L \rangle\langle \mathbf{x}_{i_1}^1, \mathbf{v} \rangle\langle \mathbf{x}_{i_3}^3, \mathbf{v} \rangle ... \langle \mathbf{x}_{i_{L-1}}^{L-1}, \mathbf{v} \rangle\langle \mathbf{x}_{i_L}^L, \mathbf{v} \rangle^2$$
$$= \mathbf{X}_1 \boldsymbol{\Lambda} \mathbf{X}_1^\top \mathbf{X}_2 \mathbf{M X}_2^\top$$
where
$$\boldsymbol{\Lambda} = \text{diag}\left(\lambda_1^1\langle \mathbf{x}_1^1, \mathbf{v} \rangle,...,\lambda_{r_{\mathcal{X}_1}}^1\langle \mathbf{x}_{r_{\mathcal{X}_1}}^1, \mathbf{v} \rangle,0,...,0\right) \hspace{1.5cm} \mathbf{M} = \text{diag}\left(\mu_1,...,\mu_{r_{\mathcal{X}_2}},0,...,0\right)$$
and
$$\mu_j = \sum_{i_3 = 1}^{r_{\mathcal{X}_3}}...\sum_{i_L = 1}^{r_{\mathcal{X}_L}} \lambda_j^2\lambda_{i_3}^3 ... \lambda_{i_L}^L \langle \mathbf{x}_{i_2}^2, \mathbf{x}_{i_3}^3 \rangle ... \langle \mathbf{x}_{i_{L-1}}^{L-1}, \mathbf{x}_{i_L}^L \rangle\langle \mathbf{x}_{i_3}^3, \mathbf{v} \rangle ... \langle \mathbf{x}_{i_{L-1}}^{L-1}, \mathbf{v} \rangle\langle \mathbf{x}_{i_L}^L, \mathbf{v} \rangle^2$$
Multiplying on the left by $\mathbf{X}_1^\top$ and on the right by $\mathbf{X}_1$, we get
$$\mathbf{X}_1^\top \mathbf{TX}_1 = \boldsymbol{\Lambda} \mathbf{X}_1^\top \mathbf{X}_2 \mathbf{MX}_2^\top \mathbf{X}_1 = \boldsymbol{\Lambda} \mathbf{QMQ}^\top$$
where $\mathbf{Q} = \mathbf{X}_1^\top \mathbf{X}_2$ is orthogonal. Suppose $r_{\mathcal{X}_1} > 1$. We seek to find a diagonal  matrix $\boldsymbol{\mathcal{L}}=\text{diag}\left(\ell_1,...,\ell_{r_{\mathcal{X}_1}},1,...,1\right) \in \mathbb{R}^{n \times n}$ with which to multiply on the left of the above equation so  that the non-zero entries of $\boldsymbol{\Lambda}$ get  canceled. Since $\mathbf{QMQ}^\top$ is symmetric, the resulting matrix 
$$\text{diag}\left(\ell_1,...,\ell_{r_{\mathcal{X}_1}},1,...,1\right){\mathbf{X}_1}^\top \mathbf{T} {\mathbf{X}_1} = \boldsymbol{\mathcal{L}}{\mathbf{X}_1}^\top \mathbf{T} {\mathbf{X}_1} = \boldsymbol{\mathcal{L}}\boldsymbol{\Lambda} \mathbf{QMQ}^\top$$
will have a symmetric top-left $r_{\mathcal{X}_1} \times r_{\mathcal{X}_1}$  block. If $r_{\mathcal{X}_1} = 1$, this symmetrizing procedure is unnecessary, since a top-left $1 \times 1$ corner block of any matrix is symmetric. This condition allows us to find a linear system in the unknowns $\ell_1, \ldots, \ell_{r_{\mathcal X_1}}$ which has the following form. Let $\overline{\mathbf{T}} \in \mathbb{R}^{r_{\mathcal{X}_1} \times r_{\mathcal{X}_1}}$ be the top-left $r_{\mathcal{X}_1} \times r_{\mathcal{X}_1}$ corner block of the matrix ${\mathbf{X}_1}^\top \mathbf{T} {\mathbf{X}_1}$. Form the matrix $\pmb{\mathscr{L}} \in \mathbb{R}^{{r_{\mathcal{X}_1} \choose 2} \times r_{\mathcal{X}_1}}$, where for each row we choose two indices $i,j \in [r_{\mathcal{X}_1}], \; i \neq j$, let the $i^{\text{th}}$  entry of the row equal $\overline{\mathbf{T}}_{i,j}$, the $j^{\text{th}}$  entry of the row equal $-\overline{\mathbf{T}}_{j,i}$, and the remaining  entries be $0$. Then we have the following result:
\begin{proposition}\label{prop:nullspace}
If $\mathcal{T}$ is generated by orthogonally decomposable carriages whose vectors and coefficients are generic, then $\text{nullsp}(\pmb{\mathscr{L}}) = \text{span}((\lambda_1^{-1},...,\lambda_{r_{\mathcal{X}_1}}^{-1}))$, which gives us precisely $\ell_1,\ldots, \ell_{r_{\mathcal X_1}}$.
\end{proposition}
\begin{proof}
Let $\mathbf{R} \in \mathbb{R}^{r_{\mathcal{X}_1} \times r_{\mathcal{X}_1}}$ be the top-left $r_{\mathcal{X}_1} \times r_{\mathcal{X}_1}$ corner block of $\mathbf{QMQ}^\top$, which is symmetric: $\mathbf{R} = \mathbf{R}^\top$. Then by the definition of $\pmb{\mathscr{L}}$, $(\ell_1, ..., \ell_{r_{\mathcal{X}_1}}) \in \text{nullsp}(\pmb{\mathscr{L}})$ satisfies
$$\text{diag}\left(\ell_1,...,\ell_{r_{\mathcal{X}_1}}\right)\text{diag}\left(\lambda_1,...,\lambda_{r_{\mathcal{X}_1}}\right)\mathbf{R} = \left(\text{diag}\left(\ell_1,...,\ell_{r_{\mathcal{X}_1}}\right)\text{diag}\left(\lambda_1,...,\lambda_{r_{\mathcal{X}_1}}\right)\mathbf{R}\right)^\top$$
$$=\mathbf{R}^\top\text{diag}\left(\lambda_1,...,\lambda_{r_{\mathcal{X}_1}}\right)\text{diag}\left(\ell_1,...,\ell_{r_{\mathcal{X}_1}}\right) = \mathbf{R}\text{diag}\left(\lambda_1,...,\lambda_{r_{\mathcal{X}_1}}\right)\text{diag}\left(\ell_1,...,\ell_{r_{\mathcal{X}_1}}\right)$$
This means that $\ell_i\lambda_i \mathbf{R}_{ij} = \ell_j\lambda_j \mathbf{R}_{ij}$ for all $i,j \in [r_{\mathcal{X}_1}]$. Since we assumed that all orthonormal sets and coefficients are generic, we have $\mathbf{R}_{ij} \neq 0$. Hence, $\frac{\ell_i}{\ell_j} = \frac{\lambda_j}{\lambda_i}$. To deduce that $(\ell_1,...,\ell_{r_{\mathcal{X}_1}}) = C(\lambda_1^{-1},...,\lambda_{r_{\mathcal{X}_1}}^{-1})$ for some constant $C \in \mathbb{R}$, we prove the following lemma:
\begin{lemma}
If $\{\ell_i\}_{i = 1}^n, \{\lambda_i\}_{i = 1}^n \subset \mathbb{R}\setminus\{0\}$ are such that for all $i,j \in [n]$, $\frac{\ell_i}{\ell_j} = \frac{\lambda_j}{\lambda_i}$, then $\ell_i = \frac{C}{\lambda_i}$, for some constant $C \in \mathbb{R}$.
\end{lemma}
\textit{Proof.} We proceed by induction on $n$. When $n = 1$, the condition $\frac{\ell_1}{\ell_1} = 1 = \frac{\lambda_1}{\lambda_1}$ is vacuous and there must exist a $C \in \mathbb{R}$ such that $\ell_1 = \frac{C}{\lambda_1}$. Now suppose the statement is true for a sets of size $n$. Given $\{\ell_i\}_{i = 1}^{n+1}$ and $\{\lambda_i\}_{i = 1}^{n+1}$ satisfying the conditions, we can apply the inductive hypothesis to the subsets $\{\ell_i\}_{i = 1}^n$ and $\{\lambda_i\}_{i = 1}^n$ to conclude that for all $i \in [n]$, $\ell_i = \frac{C}{\lambda_i}$. Then for any $i$, we have $\frac{\ell_{n+1}}{\ell_i} = \frac{\lambda_i}{\lambda_{n+1}} = \frac{\frac{C}{\ell_i}}{\lambda_{n+1}}$ and hence $\ell_{n+1} = \frac{C}{\lambda_{n+1}}$.
\end{proof}
\begin{example}
In this example we remark that Proposition~\ref{prop:nullspace} can fail to hold if the orthonormal sets are not assumed to be generic: Let $L = n = 3$, $r_{\mathcal{X}_1} = 2$, $r_{\mathcal{X}_2} = 1$, $\{\mathbf{x}_1^1,\mathbf{x}_2^1\} = \{\mathbf{e}_1,\mathbf{e}_2\}$, $\mathbf{x}_1^2 = \mathbf{e}_1$ where $\mathbf{e}_1, \mathbf{e}_2 \in \mathbb{R}^3$ are the first and second standard basis vectors, and let $\nu_1, \nu_2, \mu_1 \in \mathbb{R}$. Completing $\mathbf{x}_{i_1}^1$ and $\mathbf{x}_{i_2}^2$ to orthonormal bases, we have $\mathbf{X}_1 = \begin{pmatrix}
1 & 0 & 0 \\
0 & 1 & 0 \\
0 & 0 & \pm 1
\end{pmatrix}$ and $\mathbf{X}_2 = \begin{pmatrix}
1 & 0 & 0 \\
0 & b_1 & -b_2 \\
0 & b_2 & b_1
\end{pmatrix}$, for some $b_1, b_2 \in \mathbb{R}$ such that $b_1^2 + b_2^2 = 1$. Hence, $\boldsymbol{\Lambda} {\mathbf{X}_1}^\top \mathbf{B M B}^\top {\mathbf{X}_1} = \begin{pmatrix}
\nu_1 & 0 & 0 \\
0 & \nu_2 & 0 \\
0 & 0 & 0
\end{pmatrix}\begin{pmatrix}
1 & 0 & 0 \\
0 & b_1 & -b_2 \\
0 & \pm b_2 & \pm b_1
\end{pmatrix}\begin{pmatrix}
\mu_1 & 0 & 0 \\
0 & 0 & 0 \\
0 & 0 & 0
\end{pmatrix}\begin{pmatrix}
1 & 0 & 0 \\
0 & b_1 & \pm b_2 \\
0 & -b_2 & \pm b_1
\end{pmatrix} = \begin{pmatrix}
\nu_1\mu_1 & 0 & 0 \\
0 & 0 & 0 \\
0 & 0 & 0
\end{pmatrix}$. From the top-left $r_{\mathcal{X}_1} \times r_{\mathcal{X}_1} = 2 \times 2$ corner block $\overline{\mathbf{T}} = \begin{pmatrix}
\nu_1\mu_1 & 0 \\
0 & 0
\end{pmatrix}$ of this matrix, we construct $\pmb{\mathscr{L}} = \begin{pmatrix} 0 & 0 \end{pmatrix}$, which has nullity $2 > 1$.

\begin{example}Dropping the assumption that the coefficients are generic can also result in the proposition failing: If $L = n = 3 = r_{\mathcal{X}_1} = r_{\mathcal{X}_2} = 3$, $\{\mathbf{x}_{i_1}^1\}_{i_1 = 1}^3$ and $\{\mathbf{x}_{i_2}^2\}_{i_2 = 1}^3$ are generic orthonormal sets, $\mathbf{X}_1, \mathbf{X}_2 \in \mathbb{R}^{3 \times 3}$ have columns $\mathbf{x}_{i_1}^1$ and $\mathbf{x}_{i_2}^2$ respectively, and $\nu_{i_1} = \mu_{i_2} = 1$, then $\boldsymbol{\Lambda} \mathbf{X}_1^\top \mathbf{X}_2 \mathbf{M}  \mathbf{X}_2^\top \mathbf{X}_1 = \mathbf{I}_{3 \times 3} \mathbf{X}_1^\top \mathbf{X}_2 \mathbf{I}_{3 \times 3} \mathbf{X}_2^\top \mathbf{X}_1 = \mathbf{I}_{3 \times 3}$ whose top-left $r_{\mathcal{X}_1} \times r_{\mathcal{X}_1} = 3 \times 3$ corner block is $\overline{\mathbf{T}} = \mathbf{I}_{3 \times 3}$. Then $\pmb{\mathscr{L}} = \mathbf{0}_{3 \times 3}$ which has nullity $3 > 1$.
\end{example}

Note that while these are valid counterexamples, $\pmb{\mathscr{L}}$ having nullity greater than $1$ is not an issue in practice since any vector in the nullspace can be used for the symmetrizing procedure.
\end{example}

Thus, we let $\ell_i$ be the entries of a non-zero vector in $\text{nullsp}(\pmb{\mathscr{L}})$. It follows that the top-left $r_{\mathcal{X}_1} \times r_{\mathcal{X}_1}$ corner blocks of $\boldsymbol{\mathcal{L}} {\mathbf{X}_1}^\top \mathbf{T}\mathbf{X}_1$ and $\mathbf{Q}\widetilde{\mathbf{M}}\mathbf{Q}^\top$ are equal, where $\widetilde{\mathbf{M}} = C\mathbf{M}$ for some constant $C \in \mathbb{R}$. In fact, since multiplication on the left by a diagonal matrix corresponds to scaling of rows, this means that the first $r_{\mathcal{X}_1}$ rows of $\boldsymbol{\mathcal{L}}{\mathbf{X}_1}^\top \mathbf{T}\mathbf{X}_1$ and $\mathbf{Q}\widetilde{\mathbf{M}}\mathbf{Q}^\top$ are equal. Note that the entries in rows $r_{\mathcal{X}_1}+1$ to $n$ are all $0$. By the symmetry of $\mathbf{Q}\widetilde{\mathbf{M}}\mathbf{Q}^\top$, we can find the first $r_{\mathcal{X}_1}$ columns of $\mathbf{Q}\widetilde{\mathbf{M}}\mathbf{Q}^\top$. Hence, we have a symmetric $\mathbf{S} \in \mathbb{R}^{n \times n}$ with a bottom-right $(n - r_{\mathcal{X}_1}) \times (n - r_{\mathcal{X}_1})$ corner block of $0$'s
$$\mathbf{S} = \begin{pmatrix}
s_{11} & & \hdots & & & s_{1n} \\
\vdots & & & & & \vdots \\
s_{r_{\mathcal{X}_1}1} & & \hdots & & & s_{r_{\mathcal{X}_1},n} \\
s_{r_{\mathcal{X}_1}+1,1} & \hdots & s_{r_{\mathcal{X}_1}+1,r_{\mathcal{X}_1}} & 0 & \hdots & 0 \\
 \vdots & & \vdots & \vdots & & \vdots \\
 s_{n1} & \hdots & s_{n,r_{\mathcal{X}_1}} & 0 & \hdots &  0
\end{pmatrix}.$$
The non-zero entries of $\mathbf{S}$ and $\mathbf{Q}\widetilde{\mathbf{M}}\mathbf{Q}^\top$ are equal. The next step is to determine what entries $\widetilde{s}_{ij}$ should be filled in block of $0$'s in $\mathbf{S}$ such that the resulting matrix $\widetilde{\mathbf{S}}$ is equal to $\mathbf{Q}\widetilde{\mathbf{M}}\mathbf{Q}^\top$. Now we use the DRC: Suppose $r_{\mathcal{X}_1} \geq r_{\mathcal{X}_2}$. Then after performing Gaussian elimination on $\widetilde{\mathbf{S}}$, $\widetilde{s}_{r_{\mathcal{X}_1}+1,i} = 0$ for $r_{\mathcal{X}_1}+1 \leq i \leq n$. Thus we replace these entries in $\mathbf{S}$ with variables $x_i$
and perform Gaussian elimination. This will result in a linear expression for each $x_i$. Setting them equal to $0$, we solve for the unique values $\widetilde{s}_{r_{\mathcal{X}_1}+1,i}$.
In addition, since $\widetilde{\mathbf{S}}$ is symmetric, we know the values $\widetilde{s}_{i,r_{\mathcal{X}_1}+1}$.
$$\widetilde{\mathbf{S}}^{(1)} = \begin{pmatrix}
s_{11} & & & \hdots & & & s_{1n} \\
\vdots & & & & & & \vdots \\
s_{r_{\mathcal{X}_1}1} & & & \hdots & & & s_{r_{\mathcal{X}_1},n} \\
s_{r_{\mathcal{X}_1}+1,1} & \hdots & s_{r_{\mathcal{X}_1}+1,r_{\mathcal{X}_1}} & \widetilde{s}_{r_{\mathcal{X}_1}+1,r_{\mathcal{X}_1}+1} & \widetilde{s}_{r_{\mathcal{X}_1}+1,r_{\mathcal{X}_1}+2} & \hdots & \widetilde{s}_{r_{\mathcal{X}_1}+1,n} \\
s_{r_{\mathcal{X}_1}+2,1} & \hdots & s_{r_{\mathcal{X}_1}+2,r_{\mathcal{X}_1}} & \widetilde{s}_{r_{\mathcal{X}_1}+1,r_{\mathcal{X}_1}+2} & 0 & \hdots & 0 \\
 \vdots & & \vdots & \vdots & \vdots & & \vdots \\
 s_{n1} & \hdots & s_{n,r_{\mathcal{X}_1}} & \widetilde{s}_{r_{\mathcal{X}_1}+1,n} & 0 & \hdots & 0
\end{pmatrix}$$
We then repeat this procedure $n - r_{\mathcal{X}_1}$ more times. The final iteration will be $\widetilde{\mathbf{S}}^{(n - (r_{\mathcal{X}_1}+1))} = \widetilde{\mathbf{S}}$. Since $\mathbf{Q}\widetilde{\mathbf{M}}\mathbf{Q}^\top = {\mathbf{X}_1}^\top \mathbf{X}_2 \widetilde{\mathbf{M}} \mathbf{X}_2^\top {\mathbf{X}_1}$ is such a matrix satisfying this Gaussian elimination property when $r_{\mathcal{X}_1} \geq r_{\mathcal{X}_2}$, it follows from the uniqueness of $\tilde{s}_{ij}$ that $\widetilde{\mathbf{S}} = {\mathbf{X}_1}^\top \mathbf{X}_2 \widetilde{\mathbf{M}} \mathbf{X}_2^\top {\mathbf{X}_1}$. Thus we can obtain $\mathbf{x}_{i_2}^2$ UTPS via an eigendecomposition on ${\mathbf{X}_1} \widetilde{\mathbf{S}} {\mathbf{X}_1}^\top = \mathbf{X}_2\widetilde{\mathbf{M}} \mathbf{X}_2^\top$. If, however, $r_{\mathcal{X}_2} > r_{\mathcal{X}_1}$, then we can still apply this procedure to obtain an orthonormal set. In either case, we denote the set as $\{\mathbf{x}_{i_2}^{2,\text{LR}}\}_{i_2 = 1}^{r_{\mathcal{X}_2,\text{LR}}}$. This completes the description of the kernel completion method.

The set $\{\mathbf{x}_{i_2}^{2,\text{LR}}\}_{i_2 = 1}^{r_{\mathcal{X}_2,\text{LR}}}$ may differ from the  set $\{\mathbf{x}_{i_2}^2\}$, even UTPS. This could happen if $r_{\mathcal{X}_2,\text{LR}} > r_{\mathcal{X}_1}$, since then $r_{\mathcal{X}_2} \leq r_{\mathcal{X}_3} \leq ... \leq r_{\mathcal{X}_L}$ is true in the DRC. In this case, the direction of the decomposition (from left to right) is incorrect and we terminate the decomposition. All other sets $\{\mathbf{x}_{i_j}^{j,\text{LR}}\}_{i_j = 1}^{r_{\mathcal{X}_j,\text{LR}}}$ for $3 \leq j \leq L-1$ are assigned to be empty. Otherwise, there is no knowing whether $\{\mathbf{x}_{i_2}^{2,\text{LR}}\}_{i_2 = 1}^{r_{\mathcal{X}_2,\text{LR}}}$ is the correct set and we continue to sequentially apply kernel completion. For our example, the next iteration of kernel completion to find the set $\{\mathbf{x}_{i_3}^3\}$ would be applied to $\mathbf{T} = \mathcal{T}(\mathbf{v},\mathbf{v},\cdot,\cdot,\mathbf{v},...,\mathbf{v})$ where $\mathbf{v} \in \mathbb{R}^n$ is a generic vector. We check on each iteration whether $r_{\mathcal{X}_j,\text{LR}} > r_{\mathcal{X}_{j-1},\text{LR}}$ and take the necessary course of action. The decomposition finishes by either terminating at some point or obtaining the orthonormal set $\{\mathbf{x}_{i_{L-1}}^{L-1,\text{LR}}\}_{i_{L-1} = 1}^{r_{\mathcal{X}_{L-1},\text{LR}}}$.

{\textbf{Step 2: Decomposing the train from right to left}}

This step is the same as Step 1 but with ``LR'' interchanged with ``RL'' and $\mathcal{X}_1$ interchanged with $\mathcal{X}_L$, $\mathcal{X}_2$ interchanged with $\mathcal{X}_{L-1}$, etc.

{\textbf{Step 3: Choosing the correct orthonormal sets from the two decompositions}}

If for any $2 \leq j \leq L-1$, $\{\mathbf{x}_{i_j}^{j,\text{LR}}\}_{i_j = 1}^{r_{\mathcal{X}_j,\text{LR}}}$ and $\{\mathbf{x}_{i_j}^{j,\text{RL}}\}_{i_j = 1}^{r_{\mathcal{X}_j,\text{RL}}}$ have different ranks, then whichever has the higher rank is the correct set UTPS. The potential issue is choosing a correct set when $r_{\mathcal{X}_j,\text{LR}} = r_{\mathcal{X}_j,\text{RL}}$, for one of the sets might be incorrect. This, however, cannot happen; if, without loss of generality, $\{\mathbf{x}_{i_j}^{j,\text{LR}}\}_{i_j = 1}^{r_{\mathcal{X}_j,\text{LR}}}$ is incorrect, then it would have been because $r_{\mathcal{X}_j} > r_{\mathcal{X}_{j^\prime}}$, for some $1 \leq j^\prime < j$. But then by the DRC, we would have observed that $r_{\mathcal{X}_j} = r_{\mathcal{X}_j,\text{RL}} > r_{\mathcal{X}_j,\text{LR}}$, contradicting $r_{\mathcal{X}_j,\text{LR}} = r_{\mathcal{X}_j,\text{RL}}$.

{\textbf{Step 4: Finding the coefficients}}

Since the tensor $\mathcal{R}$ defined in~\eqref{eq:Rtensor} is a rank-$1$ tensor, we can apply a tensor decomposition algorithm such as the Alternating Least Squares (ALS) algorithm presented in \cite{KoBa09} to find the coefficients. We have observed in simulations that the rank-$1$ ALS algorithm always converges in one step. The coefficients are, of course, not unique since if $\nu_1,...,\nu_L \in \mathbb{R}\setminus\{0\}$ are non-units, then
$$\boldsymbol\lambda_1 \otimes ... \otimes \boldsymbol\lambda_L = \left(\nu_1...\nu_L\boldsymbol\lambda_1\right)\otimes\left(\frac{1}{\nu_1}\boldsymbol\lambda_2\right)\otimes...\otimes\left(\frac{1}{\nu_L}\boldsymbol\lambda_L\right)$$
is another rank-$1$ decomposition of $\mathcal{R}$.

This completes the description of our algorithm for decomposing tensors according to  an  orthogonal tensor train of length $L\geq  3$ with symmetric and orthogonal carriages. The pseudocode can be found in Algorithm~\ref{symm-odeco-alg}, and our simulation results are in Section~\ref{sec:7}. Notice that if $L = 1$, then Algorithm \ref{symm-odeco-alg} is none other than the slice method of Kolda \cite{kolda2015symmetric}, and if $L = 2$, then Algorithm \ref{symm-odeco-alg} reduces to Algorithm \ref{symmL2_alg} for symmetric orthogonal tensor trains. We have distinguished these two algorithms since the whitening procedure can be applied to symmetric but non-orthogonal tensor trains of length 2. Extending this procedure to longer trains is one of our open problems in Section \ref{sec:8}. From our analysis in this section, we have deduced the following result:

\begin{theorem}
Let $L \geq 1$. If $\mathcal{T} \in \mathbb{R}^{n^{L+2}}$ is an $(L+2)$-tensor generated by carriages $\displaystyle \mathcal{X}_j = \sum_{i_j = 1}^{r_{\mathcal{X}_j}} \lambda_{i_j}^j(\mathbf{x}_{i_j}^j)^{\otimes 3}$ where $\lambda_{i_j}^j \in \mathbb{R}$ are generic coefficients and $\{\mathbf{x}_{i_j}^j\}_{i_j = 1}^{r_{\mathcal{X}_j}} \subset \mathbb{R}^n$ are generic orthonormal sets of rank $r_{\mathcal{X}_j}$ for all $j \in [L]$, and $\mathcal{T}$ satisfies the Decreasing Ranks Condition (c.f. Definition \ref{drc}), then Algorithm \ref{symm-odeco-alg} recovers all of the ranks of $\mathcal{T}$, and all of the vectors and coefficients of $\mathcal{T}$ UTPS.
\end{theorem}

\begin{algorithm}
\caption{Orthogonal Symmetric Decomposition for Tensor Trains of Length $L \geq 3$}\label{symm-odeco-alg}
\textbf{Input:}
$\mathcal{T} \in \mathbb{R}^{n^{L+2}}$, an orthogonal symmetric tensor train of length $L \geq 3$ satisfying the DRC

\textbf{Output:}
$\boldsymbol{\lambda}_j \in \mathbb{R}^{r_{\mathcal{X}_j}},\{\mathbf{x}_{i_j}^j\}_{i_j = 1}^{r_{\mathcal{X}_j}} \subset \mathbb{R}^n$, the coefficients and orthonormal vectors which generate $\mathcal{T}$, correct UTPS in the vectors and non-zero scaling in the coefficients, for $j \in [L]$

\begin{algorithmic}[1]
\State $\mathbf{v} \gets$ generic vector in $\mathbb{R}^n$
\State $\{\sigma_{i_1}, \mathbf{x}_{i_1}^1\}_{i_1=1}^{r_{\mathcal{X}_1}} \gets $ eigenpairs of $\mathcal{T}(\cdot,\cdot,\mathbf{v},...,\mathbf{v})$ with $\sigma_{i_1} \neq 0$
\For {$j = 2,...,L-1$}
    \State $\mathbf{A} \gets$ $n \times n$ orthogonal matrix whose first $r_{\mathcal{X}_{j-1,\text{LR}}}$ columns are $\mathbf{x}_{i_{j-1}}^{j-1,\text{LR}}$
    \State $\mathbf{v} \gets$ generic vector in $\mathbb{R}^n$
    \State $\mathbf{T} \gets \mathcal{T}(\mathbf{v},...,\mathbf{v},\cdot,\cdot,\mathbf{v},...,\mathbf{v})$ where modes $j-1$ and $j$ have not been contracted
    \State $\mathbf{S} \gets \mathbf{A}^\top \mathbf{T}\mathbf{A}$
    \State $\overline{\mathbf{T}} \gets$ top-left $r_{\mathcal{X}_{j-1},\text{LR}} \times r_{\mathcal{X}_{j-1},\text{LR}}$ corner block of $\mathbf{S}$
    \If {$r_{\mathcal{X}_{j-1},\text{LR}} > 1$}
        \State \multiline{$\pmb{\mathscr{L}} \gets \mathbb{R}^{{r_{\mathcal{X}_{j-1},\text{LR}} \choose 2} \times r_{\mathcal{X}_{j-1},\text{LR}}}$ matrix where each row corresponds to a pair $(i,j)$, $i \neq j$, $i,j \in [r_{\mathcal{X}_{j-1},\text{LR}}]$, column entry $i$ and $j$ are $\overline{\mathbf{T}}_{i,j}$ and $-\overline{\mathbf{T}}_{j,i}$, and remaining entries are $0$}
        \State $(\ell_1,...,\ell_{r_{\mathcal{X}_{j-1},\text{LR}}}) \gets$ non-zero vector in $\text{nullsp}(\pmb{\mathscr{L}})$
        \State $\mathbf{S} \gets \text{diag}(\ell_1,...,\ell_{r_{\mathcal{X}_{j-1},\text{LR}}},1,...,1)\mathbf{S}$
        \State $\mathbf{S} \gets$ first $r_{\mathcal{X}_{j-1},\text{LR}}$ columns are equal to the transpose of first $r_{\mathcal{X}_{j-1},\text{LR}}$ rows of $\mathbf{S}$
    \EndIf
    \For {$k = r_{\mathcal{X}_{j-1},\text{LR}}+1,...,n$}
        \State $\mathbf{R} \gets$ replace entries $\mathbf{S}_{k,m}$ with variables $x_m$, $k \leq m \leq n$
        \State $\mathbf{R} \gets $ perform Gaussian elimination on $\mathbf{R}$
        \State $x_m \gets$ set all entries in $\mathbf{R}$ containing $x_m$ to $0$ and solve for each $x_m$
        \State $\mathbf{S} \gets$ replace entries $\mathbf{S}_{k,m}$ and $\mathbf{S}_{m,k}$ with the values $x_m$
    \EndFor
    \State $\{\sigma_{i_j}, \mathbf{x}_{i_j}^{j,\text{LR}}\}_{i_j=1}^{r_{\mathcal{X}_j,\text{LR}}} \gets $ eigenpairs of $\mathbf{A}\mathbf{S}\mathbf{A}^\top$ with $\sigma_{i_j} \neq 0$
    \If {$r_{\mathcal{X}_j,\text{LR}} > r_{\mathcal{X}_{j-1},\text{LR}}$}
        \State $\{\mathbf{x}_{i_k}^{k,\text{LR}}\}_{i_k=1}^{r_{\mathcal{X}_k,\text{LR}}} \gets \emptyset$, $j \leq k \leq L-1$
        \State \textbf{break}
    \EndIf
\EndFor
\State \multiline{Repeat steps $1$ to $26$, with ``LR'' interchanged with ``RL'', $j = 2,...,L-1$ with $j = L-1,...,2$, $j-1$ with $j+1$, step $1$ replaced with $\mathcal{T}(\mathbf{v},...,\mathbf{v},\cdot,\cdot)$, and step $7$ replaced with $\mathbf{S} \gets \mathbf{A}^\top \mathbf{T}^\top\mathbf{A}$}
\State \multiline{$\{\mathbf{x}_{i_j}^j\}_{i_j = 1}^{r_{\mathcal{X}_j}} \gets \{\mathbf{x}_{i_j}^{j,\text{LR}}\}_{i_j = 1}^{r_{\mathcal{X}_{j,\text{LR}}}},\{\mathbf{x}_{i_j}^{j,\text{RL}}\}_{i_j = 1}^{r_{\mathcal{X}_{j,\text{RL}}}}$, whichever of $r_{\mathcal{X}_{j,\text{LR}}}$ or $r_{\mathcal{X}_{j,\text{RL}}}$ is greater, $2 \leq j \leq L-1$. If they are the same, make an arbitrary choice}
\State $\mathcal{R} \gets \mathbb{R}^{r_{\mathcal{X}_1}\times...\times r_{\mathcal{X}_L}}$ tensor with entries $\mathcal{R}_{\hat{i}_1...\hat{i}_L} = \frac{\mathcal{T}(\mathbf{x}_{\hat{i}_1}^1,...,\mathbf{x}_{\hat{i}_L}^L)}{\prod_{k = 1}^{L-1}\langle \mathbf{x}_{\hat{i}_k}^k, \mathbf{x}_{\hat{i}_{k+1}}^{k+1} \rangle}$, $\hat{i}_j \in [r_{\mathcal{X}_j}]$, $j \in [L]$
\State $\boldsymbol\lambda_j \gets$ rank-1 ALS algorithm on $\mathcal{R}$, $j \in [L]$
\end{algorithmic}
\end{algorithm}

\medskip

\section{Orthogonal Decomposition of Tensor Trains of Length 2}\label{sec:5}

We now  turn to the problem of decomposing  4-tensors according to tensor trains of length 2 such that the carriages are orthogonal but not necessarily symmetric.

\begin{problem}{\label{odeco-prob}}
Let $\mathcal{T} \in \mathbb{R}^{n_\mathbf{A} \times n_\mathbf{B} \times n_\mathbf{D} \times n_\mathbf{E}}$ be a $4$-tensor admitting the following decomposition:
$$\mathcal{T} = \sum_{i=1}^{r_\mathcal{L}} \sum_{j=1}^{r_\mathcal{R}} \lambda_i \mu_j \mathbf{a}_i \otimes \mathbf{b}_i \otimes \mathbf{d}_j \otimes \mathbf{e}_j \langle \mathbf{c}_i, \mathbf{f}_j \rangle$$
where $\{\mathbf{a}_i\}_{i=1}^{r_\mathcal{L}} \subset \mathbb{R}^{n_\mathbf{A}}$, $\{\mathbf{b}_i\}_{i=1}^{r_\mathcal{L}} \subset \mathbb{R}^{n_\mathbf{B}}$,  $\{\mathbf{c}_i\}_{i=1}^{r_\mathcal{L}} \subset \mathbb{R}^{n_\mathbf{C}}$,  $\{\mathbf{d}_j\}_{j=1}^{r_\mathcal{R}} \subset \mathbb{R}^{n_\mathbf{D}}$, $\{\mathbf{e}_j\}_{j=1}^{r_\mathcal{R}} \subset \mathbb{R}^{n_\mathbf{E}}$,  $\{\mathbf{f}_j\}_{j=1}^{r_\mathcal{R}} \subset \mathbb{R}^{n_\mathbf{F}}$ are generic orthonormal sets, $\lambda_i,  \mu_j \in \mathbb{R}$ are generic, and $d = n_{\mathbf{C}} = n_{\mathbf{F}}$ (necessarily $d \geq r_\mathcal{L}, r_\mathcal{R}$ due to orthonormality). In other words, assume that $\mathcal{T}$ is a tensor train generated by
$$\mathcal{L} = \sum_{i=1}^{r_\mathcal{L}} \lambda_i \mathbf{a}_i \otimes \mathbf{b}_i \otimes \mathbf{c}_i \hspace{2cm} \mathcal{R} = \sum_{j=1}^{r_\mathcal{R}} \mu_j \mathbf{d}_j \otimes \mathbf{e}_j \otimes \mathbf{f}_j$$
Given $\mathcal T$, we wish to find all ranks, vectors, and coefficients.
\end{problem}

We first note that the vectors $\mathbf{a}_i$, $\mathbf{b}_i$, $\mathbf{d}_j$, and $\mathbf{e}_j$ can be found by adapting Kolda's slice method for  orthogonal tensors~\cite{kolda2015symmetric}. As before, define the matrices
$$\mathbf{S}_\mathcal{L} = \sum_{i_3 = 1}^n \sum_{i_4 = 1}^n \alpha_{i_3 i_4} \mathcal{T}(:, :, i_3, i_4) = \sum_{i=1}^{r_\mathcal{L}} \sigma_i \mathbf{a}_i \mathbf{b}_i^\top = \mathbf{A}\boldsymbol\Sigma\mathbf{B}^\top$$
$$\mathbf{S}_\mathcal{R} = \sum_{i_1 = 1}^n \sum_{i_2 = 1}^n \beta_{i_1 i_2} \mathcal{T}(i_1, i_2, :, :) = \sum_{j=1}^{r_\mathcal{R}} \gamma_j \mathbf{d}_j \mathbf{e}_j^\top = \mathbf{D}\boldsymbol\Gamma\mathbf{E}^\top$$
where $\boldsymbol\Sigma \in \mathbb{R}^{r_{\mathcal{L}} \times r_{\mathcal{L}}}, \boldsymbol\Gamma \in \mathbb{R}^{r_{\mathcal{R}} \times r_{\mathcal{R}}}$ are diagonal with non-zero diagonal entries, and $\mathbf{A} \in \mathbb{R}^{n_\mathbf{A} \times r_\mathcal{L}}$, $\mathbf{B} \in \mathbb{R}^{n_\mathbf{B} \times r_\mathcal{L}}$, $\mathbf{C}\in\mathbb R^{n_{\mathbf C}\times r_{\mathcal L}}$ $\mathbf{D} \in \mathbb{R}^{n_\mathbf{D} \times r_\mathcal{R}}$, $\mathbf{E} \in \mathbb{R}^{n_\mathbf{E} \times r_\mathcal{R}}$, $\mathbf F\in\mathbb R^{n_{\mathbf F}\times r_{\mathcal R}}$ are the matrices that have $\{\mathbf{a}_i\}$, $\{\mathbf{b}_i\}$, $\{\mathbf c_i\}$ $\{\mathbf{d}_j\}$, $\{\mathbf{e}_j\}$, and $\{\mathbf f_{j}\}$  as their columns, respectively. Then $\mathbf A,  \mathbf B, \mathbf D,$ and $\mathbf E$ can be found by finding the SVD of $\mathbf{S}_\mathcal{L}$ and $\mathbf{S}_\mathcal{R}$. Next we see that for each $i \in [r_{\mathcal{L}}], j \in [r_{\mathcal{R}}]$, $\mathcal{T}(\mathbf{a}_i,\mathbf{b}_i,\mathbf{d}_j,\mathbf{e}_j) = \lambda_i \mu_j \langle \mathbf{c}_i, \mathbf{f}_j \rangle$, which defines the entries of a matrix $\overline{\mathbf{X}} \in \mathbb{R}^{r_{\mathcal{L}} \times r_{\mathcal{R}}}$. Note that $n_{\mathbf C} = n_{\mathbf F}$ since they correspond to the same edge of the tensor network. Let $d = n_{\mathbf C} = n_{\mathbf F}$, and suppose we know $d$. Then we can define the 0-padded  $d\times d$ matrix
\begin{equation}{\label{x-dodd-eq}}
\mathbf{X} =
\begin{tikzpicture}[baseline={([yshift=-.5ex]current bounding box.center)},vertex/.style={anchor=base,
    circle,fill=black!25,minimum size=18pt,inner sep=2pt}]
    \matrix [matrix of math nodes,left delimiter=(,right delimiter=)] (m)
    {
        \; & \; & \; & \; & \; & \; \\ 
        \; & \overline{\mathbf{X}} & \; & \; & \; & \; \\ 
        \; & \; & \; & \; & \; & \; \\ 
        \; & \; & \; & \; & \; & \; \\ 
        \; & \; & \; & \; & \mathbf{0} & \; \\ 
        \; & \; & \; & \; & \; & \; \\ 
    };
    \draw (m-3-1.south west) -- (m-3-3.south east);
    \draw (m-3-3.south east) -- (m-1-3.north east);
\end{tikzpicture} 
=
\left(\begin{array}{cccccccc} 
\lambda_1 \mu_1 \langle \mathbf{c}_1, \mathbf{f}_1 \rangle & & \hdots & & \multicolumn{1}{c|}{\lambda_1 \mu_{r_\mathcal{R}} \langle \mathbf{c}_1, \mathbf{f}_{r_\mathcal{R}} \rangle} & 0 & \hdots & 0 \\
& & & & \multicolumn{1}{c|}{\phantom{0}} & \vdots & & \vdots \\
\vdots & & \ddots & & \multicolumn{1}{c|}{\vdots} & & & \\
& & & & \multicolumn{1}{c|}{\phantom{0}} & & & \\
\lambda_{r_\mathcal{L}} \mu_1 \langle \mathbf{c}_{r_\mathcal{L}}, \mathbf{f}_1 \rangle & & \hdots & & \multicolumn{1}{c|}{\lambda_1 \mu_{r_\mathcal{R}} \langle \mathbf{c}_1, \mathbf{f}_{r_\mathcal{R}} \rangle} & \vdots & & \\
\cmidrule{1-5}
0 & \hdots & & & \hdots & 0 & & \\
\vdots & & & & & & & \vdots \\
0 & \hdots & & & & & \hdots & 0
\end{array}\right)
\end{equation}
\begin{equation*}
= \text{diag}\left(\lambda_1,...,\lambda_{r_{\mathcal{L}}},0,...,0\right)\widehat{\mathbf{C}}^\top\widehat{\mathbf{F}}\text{diag}\left(\mu_1,...,\mu_{r_{\mathcal{R}}},0,...,0\right) = \boldsymbol{\Lambda}\widehat{\mathbf{C}}^\top\widehat{\mathbf{F}}\mathbf{M} \in \mathbb{R}^{d \times d}.   
\end{equation*}
where $\widehat{\mathbf{C}},\widehat{\mathbf{F}} \in \mathbb{R}^{d \times d}$ are orthogonal matrices whose first $r_{\mathcal{L}}$ and $r_{\mathcal{R}}$ columns are $\mathbf{C}$ and $\mathbf{F}$. To solve {Problem \ref{odeco-prob}}, we need to find $\boldsymbol\Lambda$, $\mathbf{M}$, $\widehat{\mathbf{C}}$, and $\widehat{\mathbf{F}}$ such that $\mathbf{X} = \boldsymbol{\Lambda}\widehat{\mathbf{C}}^\top\widehat{\mathbf{F}}\mathbf{M}$. If a solution exists, it is not unique since $\mathbf{C}^\prime = \widehat{\mathbf{C}}^\top\widehat{\mathbf{F}}$ and $\mathbf{F}^\prime = \mathbf{I}_{d \times d}$ is also a solution. Hence we can instead ask to find an orthogonal matrix $\mathbf{Q} = \widehat{\mathbf{C}}^\top\widehat{\mathbf{F}} \in \mathbb{R}^{d \times d}$ such that $\mathbf{X} = \boldsymbol{\Lambda}\mathbf{Q}\mathbf{M}$. Lastly, we can assume $\boldsymbol\Lambda$ and $\mathbf{M}$ have non-negative diagonal entries since if they do not, we can ``push'' the negative signs of these entries into $\mathbf{Q}$ without affecting its orthogonality. Note that even with this assumption, if a solution exists,  then $\mathbf \Lambda$  and $\mathbf M$ are still not unique due to a global non-zero scaling.

A matrix of the form $\boldsymbol{\Lambda}\mathbf{Q}\mathbf{M}$ is said to have a \textbf{Diagonal-Orthogonal-Diagonal Decomposition (DODD)}, which is the subject of Section \ref{sec:6}. Note that there could be many possible $d$ for which $\mathbf{X}$ in $(\ref{x-dodd-eq})$ has a DODD. Thus to solve {Problem \ref{odeco-prob}} (see pseudocode in Algorithm~\ref{odeco-len2-alg}), we solve the problem of finding such a $d$ and $\boldsymbol\Lambda$, $\mathbf{Q}$, and $\mathbf{M}$; this is formally stated in {Problem \ref{dodd-problem}}. When $r_\mathcal{L} = r_\mathcal{R}$ and we know that $d = r_\mathcal{L} = r_\mathcal{R}$ gives an $\mathbf{X}$ with a DODD, which we call the \textbf{square case}, we show that {Problem \ref{dodd-problem}} can be solved using two different approaches: one based on Sinkhorn's algorithm \cite{sinkhorn1967} and the other based on the Tandem Procrustes algorithm \cite{Everson97}. We call the complementary case the \textbf{general case} and we present a solution by showing that the Procrustes-based algorithm can be generalized. Lastly, similar to {Remark \ref{symm-odeco-2-remark}} in {Section \ref{sec:3}}, we add that one can easily generalize to the case where each carriage is an $m$-tensor, for some $m > 3$, but cannot easily address the case where there is more than one contracted edge between $\mathcal{L}$ and $\mathcal{R}$.

\begin{algorithm}[t]
\caption{Orthogonal Decomposition for Tensor Trains of Length 2}{\label{odeco-len2-alg}}

\textbf{Input:}

$\mathcal{T} \in \mathbb{R}^{n_\mathbf{A} \times n_\mathbf{B} \times n_\mathbf{D} \times n_\mathbf{E}}$, a length $2$ orthogonal tensor train

\textbf{Output:}

$r_\mathcal{L}$ and $r_\mathcal{R}$, the ranks of $\mathcal{L}$ and $\mathcal{R}$

$\{\mathbf{a}_i\}_{i=1}^{r_\mathcal{L}}$, $\{\mathbf{b}_i\}_{i=1}^{r_\mathcal{L}}$,  $\{\mathbf{c}_i\}_{i=1}^{r_\mathcal{L}}$,  $\{\mathbf{d}_j\}_{j=1}^{r_\mathcal{R}}$, $\{\mathbf{e}_j\}_{j=1}^{r_\mathcal{R}}$,  $\{\mathbf{f}_j\}_{j=1}^{r_\mathcal{R}}$, the orthonormal vectors generating $\mathcal{T}$

$\{\lambda_i\}_{i = 1}^{r_\mathcal{L}}$ and $\{\mu_j\}_{j = 1}^{r_\mathcal{R}}$, the coefficients generating $\mathcal{T}$
\begin{algorithmic}[1]
\State $\alpha, \beta \gets $ generic real $(n_\mathbf{A} \times n_\mathbf{B})$ and $(n_\mathbf{D} \times n_\mathbf{E})$ matrices
\State $\mathbf{S}_\mathcal{L} \gets \sum_{i_3 = 1}^{n_\mathbf{D}} \sum_{i_4 = 1}^{n_\mathbf{E}} \alpha_{i_3 i_4} {\mathcal{T}}(:, :, i_3, i_4)$
\State $\mathbf{S}_\mathcal{R} \gets \sum_{i_1 = 1}^{n_\mathbf{A}} \sum_{i_2 = 1}^{n_\mathbf{B}} \beta_{i_1 i_2} {\mathcal{T}}(i_1, i_2, :, :)$
\State $\{\mathbf{{a}}_i, \mathbf{{b}}_i\}_{i=1}^{r_\mathcal{L}} \gets $ left and right singular vectors of $\mathbf{S}_\mathcal{L}$ with nonzero singular values
\State $\{\mathbf{{d}}_j, \mathbf{{e}}_j\}_{j=1}^{r_\mathcal{R}} \gets $ left and right singular vectors of $\mathbf{S}_\mathcal{R}$ with nonzero singular values
\State $\overline{\mathbf{X}}_{ij} \gets {\mathcal{T}}({\mathbf{a}}_i,{\mathbf{b}}_i,{\mathbf{d}}_j,{\mathbf{e}}_j)$ for $i \in [r_\mathcal{L}], \; j \in [r_\mathcal{R}]$
\If {$m = n$ and we know that $d = m = n$ gives a DODD for the $d \times d$ 0-padding of $\overline{\mathbf{X}}$}
    \State $\mathbf{\Lambda}$, $\mathbf{Q}$, $\mathbf{M}$ $\gets$ \textsc{square\_dodd}$(\mathbf{T})$
\Else
    \State $d \gets $ a value at least $\text{max}\{r_{\mathcal{L}},r_{\mathcal{R}}\}$ for which the $d \times d$ 0-padding of $\overline{\mathbf{X}}$ admits a DODD
    \State $\mathbf{\Lambda}$, $\mathbf{Q}$, $\mathbf{M}$ $\gets$ \textsc{dodd}$(\mathbf{T}, d)$
\EndIf
\State $\lambda_i \gets$ the first $r_\mathcal{L}$ diagonal entries of $\boldsymbol\Lambda$ (which are non-zero)
\State $\mu_i \gets$ the first $r_\mathcal{R}$ diagonal entries of $\mathbf{M}$ (which are non-zero)
\State $\mathbf{c}_i \gets$ the first $r_\mathcal{L}$ rows of $\mathbf{Q}$
\State $\mathbf{f}_j \gets$ the first $r_\mathcal{R}$ standard basis vectors of $\mathbb{R}^d$
\end{algorithmic}
\end{algorithm}

{

\subsection{Applications}{\label{sec::5.1}}

The decomposition of length-2 tensor trains is closely related to the tensor hypercontraction (THC) method of compressing the electron repulsion integral (ERI) tensor common in electronic structure theory, first described in \cite{hohenstein2012} and \cite{parrish2012}. One seeks an approximate factorization of the fourth-order ERI tensor $\mathcal{R}$ of the form

\begin{equation}{\label{eri}}
\mathcal{R}_{\mu\nu\lambda\sigma} = \sum_{P,Q} (\mathbf{x}_P)_{\mu} (\mathbf{x}_P)_{\nu} \mathbf{Z}_{P,Q} (\mathbf{x}_Q)_{\lambda} (\mathbf{x}_Q)_{\sigma}    
\end{equation}
where $\mathbf{x}_P$ are vectors and $\mathbf{Z}$ is known as the intermediate matrix. To find this, the PARAFAC-THC method is proposed, which first finds a factorization of $\mathcal{R}$ as the contraction of two third-order tensors $\mathcal{S}$ and $\mathcal{T}$, and three matrices $\mathbf{L}$, $\mathbf{M}$, and $\mathbf{N}$, as shown in Figure \ref{electron}. Using the Alternating Least Squares algorithm \cite{KoBa09}, a factorization

$$\mathcal{T} = \sum_{P} \mathbf{x}_P \otimes \mathbf{x}_P \otimes \mathbf{y}_P \quad \quad \quad \quad \quad \mathcal{S} = \sum_{Q} \mathbf{x}_Q \otimes \mathbf{x}_Q \otimes \mathbf{y}_Q$$
is obtained, and setting $\mathbf{Z}_{P,Q} = \sum_{A,B,C,D} (\mathbf{y}_P)_A \mathbf{L}_{A,B}\mathbf{M}_{B,C}\mathbf{N}_{C,D}(\mathbf{y}_Q)_D$ completes the factorization (\ref{eri}).

Now suppose $\{\mathbf{x}_P\}_P$ is an orthonormal set. Then by fully contracting the tensor network in Figure \ref{electron} into a length-2 tensor train, we see that the vectors $\mathbf{x}_P$ can also be approximated (or in the square case, found exactly) using Algorithm \ref{odeco-len2-alg}. Following this, the values of $\mathbf{Z}_{P,Q}$ can easily be recovered without knowing $\mathbf{L}$, $\mathbf{M}$, or $\mathbf{N}$. \cite{lu2015compression} also proposes another approximate solution using Fast Fourier Transforms.

}

\begin{figure}[H]
\begin{center}

\tikzset{every picture/.style={line width=0.75pt}} 

\begin{tikzpicture}[x=0.75pt,y=0.75pt,yscale=-1,xscale=1]

\draw    (140.49,59.36) -- (199.16,59.36) ;
\draw  [fill={rgb, 255:red, 0; green, 0; blue, 0 }  ,fill opacity=1 ] (196.88,59.36) .. controls (196.88,58.04) and (197.9,56.97) .. (199.16,56.97) .. controls (200.41,56.97) and (201.43,58.04) .. (201.43,59.36) .. controls (201.43,60.67) and (200.41,61.74) .. (199.16,61.74) .. controls (197.9,61.74) and (196.88,60.67) .. (196.88,59.36) -- cycle ;
\draw    (199.16,59.36) -- (257.82,59.36) ;

\draw    (275.19,59.22) -- (333.86,59.22) ;
\draw  [fill={rgb, 255:red, 0; green, 0; blue, 0 }  ,fill opacity=1 ] (331.58,59.22) .. controls (331.58,57.9) and (332.6,56.83) .. (333.86,56.83) .. controls (335.11,56.83) and (336.13,57.9) .. (336.13,59.22) .. controls (336.13,60.54) and (335.11,61.61) .. (333.86,61.61) .. controls (332.6,61.61) and (331.58,60.54) .. (331.58,59.22) -- cycle ;
\draw    (333.86,59.22) -- (392.52,59.22) ;

\draw    (410.59,59.02) -- (469.26,59.02) ;
\draw  [fill={rgb, 255:red, 0; green, 0; blue, 0 }  ,fill opacity=1 ] (466.98,59.02) .. controls (466.98,57.7) and (468,56.63) .. (469.26,56.63) .. controls (470.51,56.63) and (471.53,57.7) .. (471.53,59.02) .. controls (471.53,60.34) and (470.51,61.41) .. (469.26,61.41) .. controls (468,61.41) and (466.98,60.34) .. (466.98,59.02) -- cycle ;
\draw    (469.26,59.02) -- (527.92,59.02) ;

\draw    (656.33,97) -- (606.25,59.46) ;
\draw    (606.25,59.46) -- (656.33,28.2) ;
\draw    (606.25,59.46) -- (547.59,58.86) ;
\draw  [fill={rgb, 255:red, 0; green, 0; blue, 0 }  ,fill opacity=1 ] (609.8,59.5) .. controls (609.79,60.82) and (608.76,61.87) .. (607.5,61.86) .. controls (606.24,61.85) and (605.24,60.77) .. (605.25,59.45) .. controls (605.26,58.13) and (606.29,57.07) .. (607.55,57.08) .. controls (608.81,57.1) and (609.82,58.18) .. (609.8,59.5) -- cycle ;

\draw    (14.13,21.69) -- (64.28,59.15) ;
\draw    (64.28,59.15) -- (14.25,90.49) ;
\draw    (64.28,59.15) -- (122.94,59.65) ;
\draw  [fill={rgb, 255:red, 0; green, 0; blue, 0 }  ,fill opacity=1 ] (60.72,59.12) .. controls (60.73,57.8) and (61.76,56.74) .. (63.02,56.75) .. controls (64.28,56.76) and (65.29,57.84) .. (65.28,59.16) .. controls (65.26,60.48) and (64.24,61.54) .. (62.98,61.53) .. controls (61.72,61.51) and (60.71,60.44) .. (60.72,59.12) -- cycle ;

\draw (40.4,23) node [anchor=north west][inner sep=0.75pt]  [font=\small]  {$\mathbf{x}_{P}$};
\draw (41.22,76.42) node [anchor=north west][inner sep=0.75pt]  [font=\small]  {$\mathbf{x_{P}}$};
\draw (78.4,37.2) node [anchor=north west][inner sep=0.75pt]  [font=\small]  {$\mathbf{y_{P}}$};
\draw (602.4,22.2) node [anchor=north west][inner sep=0.75pt]  [font=\small]  {$\mathbf{x_{Q}}$};
\draw (603.4,76.6) node [anchor=north west][inner sep=0.75pt]  [font=\small]  {$\mathbf{x_{Q}}$};
\draw (566.8,38.2) node [anchor=north west][inner sep=0.75pt]  [font=\small]  {$\mathbf{y_{Q}}$};
\draw (194.4,38.4) node [anchor=north west][inner sep=0.75pt]  [font=\small]  {$\mathbf{L}$};
\draw (326.6,37.6) node [anchor=north west][inner sep=0.75pt]  [font=\small]  {$\mathbf{M}$};
\draw (462.8,37.6) node [anchor=north west][inner sep=0.75pt]  [font=\small]  {$\mathbf{N}$};

\end{tikzpicture}

\end{center}
\caption{{{A PARAFAC-THC decomposition \cite{parrish2012} of the ERI tensor.}}}
\label{electron}
\end{figure}
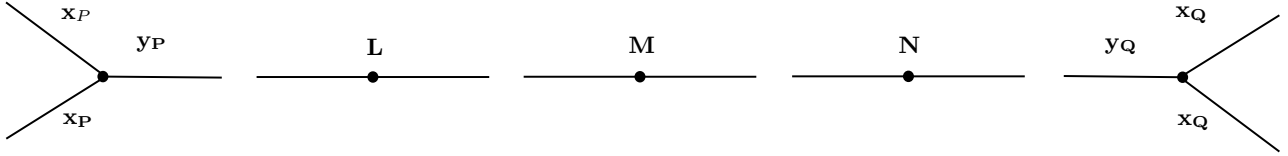

\medskip

\section{Matrix Diagonal-Orthogonal-Diagonal Decomposition}
\label{sec:6}

We now discuss the problem of finding a diagonal-orthogonal-diagonal decomposition of a given matrix. Let $m = r_\mathcal{L}$ and $n = r_\mathcal{R}$, from the previous section.

\begin{problem}{\label{dodd-problem}}
Find $d \geq m,n$ such that $\mathbf{X}$, the $d\times d$ 0-padding of the matrix $\overline{\mathbf{X}}$ in $(\ref{x-dodd-eq})$, admits a decomposition of the form $$\mathbf{X} = \text{diag}\left(\lambda_1,...,\lambda_{m},0,...,0\right)\mathbf{Q}\text{diag}\left(\mu_1,...,\mu_n,0,...,0\right) = \boldsymbol{\Lambda} \mathbf{Q}\mathbf{M} \in \mathbb{R}^{d \times d}$$
where $\mathbf{Q} \in \mathbb{R}^{d \times d}$ is an orthogonal matrix, and $\lambda_i,\mu_j \in \mathbb{R}$. Given such an $\mathbf{X}$, find $\boldsymbol{\Lambda}$, $\mathbf{Q}$, and $\mathbf{M}$.
\end{problem}

\subsection{The Square Case: A Sinkhorn-Based Algorithm}

We first solve the case when $m = n$ and we know that $\overline{\mathbf{X}}$ has a DODD for $d = m = n$. We do this by modifying Sinkhorn's algorithm \cite{sinkhorn1967}, and we additionally require that $\mathbf{X}$ has no entries equal to 0. Let $\cdot^{*2}$ and $\sqrt[*]{\cdot}$ denote the entry-wise square and square root of a matrix, and let $\odot$ denote the Hadamard product of two matrices. Then, $\mathbf X^{*2} = \mathbf \Lambda^2 \mathbf Q^{*2}\mathbf M^2$.
Every entry of $\mathbf{X}^{*2}$ is positive, and since $\mathbf{Q}$ is orthogonal, $\mathbf{Q}^{*2}$ is a doubly stochastic matrix. Sinkhorn's algorithm~\cite{sinkhorn1967} applied to $\mathbf{X}^{*2}$ first rescales all rows of the matrix $\mathbf X^{*2}$ so they each sum to 1, then it rescales all columns  so  they sum to  1, and then the rows, alternating until convergence. The total rescaling of rows and columns is recorded in two diagonal matrices, $\widetilde{\boldsymbol{\Lambda}}$ and $\widetilde{\mathbf{M}}$, respectively.
Sinkhorn's Theorem \cite{sinkhorn1967} guarantees that this algorithm converges and yields the unique doubly stochastic matrix $\mathbf Q^{*2}$ and two positive diagonal matrices $\widetilde{\boldsymbol{\Lambda}}$ and $\widetilde{\mathbf{M}}$, unique up to non-zero, positive scaling, such that $\mathbf X^{*2} = \widetilde{\boldsymbol\Lambda} \mathbf Q^{*2}\widetilde{\mathbf M}$. Hence, we obtain $\mathbf{Q}^{*2}={\widetilde{\boldsymbol{\Lambda}}}^{-1}\mathbf{X}^{*2}{\widetilde{\mathbf{M}}}^{-1}$. The absolute values of the entries of $\mathbf{Q}$ are therefore $\sqrt[*]{\mathbf{Q}^{*2}}$. Since $\widetilde{\boldsymbol{\Lambda}}$ and $\widetilde{\mathbf{M}}$ are positive diagonal matrices, the sign modifications to the entries of $\sqrt[*]{\mathbf{Q}^{*2}}$ required for it to be orthogonal can be obtained directly from the sign pattern of $\mathbf{X}$. That is, if $\mathbf{S} \in 
\mathbb{R}^{d \times d}$ is the matrix such that $\mathbf{S}_{ij} = 1$ if $\mathbf{X}_{ij} > 0$ and $\mathbf{S}_{ij} = -1$ if $\mathbf{X}_{ij} < 0$, then $\mathbf{S} \odot \sqrt[*]{\mathbf{Q}^{*2}}$ is an orthogonal matrix. Lastly, taking the entry-wise square root of both sides of $\widetilde{\boldsymbol{\Lambda}}\mathbf{X}^{*2}\widetilde{\mathbf{M}} = \mathbf{Q}^{*2}$ together with the sign matrix $\mathbf{S}$, it follows that $\mathbf{X}= \left(\sqrt[*]{\widetilde{\boldsymbol{\Lambda}}}\right)^{-1}\left(\mathbf{S} \odot \sqrt[*]{\mathbf{Q}^{*2}}\right)\left(\sqrt[*]{\widetilde{\mathbf{M}}}\right)^{-1}$. We present pseudocode in {Algorithm \ref{sinkhorn_dodd}}. From this discussion, we have also obtained the following result:

\begin{theorem}
If $\overline{\mathbf{X}} \in \mathbb{R}^{n \times n}$ admits a decomposition of the form $\overline{\mathbf{X}} = \boldsymbol\Lambda\mathbf{Q}\mathbf{M}$ where $\boldsymbol\Lambda, \mathbf{M} \in \mathbb{R}^{n \times n}$ are diagonal matrices and $\mathbf{Q} \in \mathbb{R}^{n \times n}$ is an orthogonal matrix, and $\mathbf{X}$ has no entries equal to $0$, then Algorithm \ref{sinkhorn_dodd} is guaranteed to find $\boldsymbol\Lambda$, $\mathbf{Q}$, and $\mathbf{M}$, UTPS.
\end{theorem}

\begin{algorithm}
\caption{Sinkhorn-Based Algorithm for Square DODD}\label{sinkhorn_dodd}
\textbf{Input:} 

$\mathbf{X} \in \mathbb{R}^{d \times d}$, a matrix with an existing DODD

\textbf{Output:} 

Diagonal matrices $\boldsymbol{\Lambda}, \mathbf{M} \in \mathbb{R}^{d \times d}$ and orthogonal matrix $\mathbf{Q} \in \mathbb{R}^{d \times d}$ such that $\mathbf{X} = \boldsymbol{\Lambda} \mathbf{Q} \mathbf{M}$

\textbf{Algorithm:}

\begin{algorithmic}[1]
\Procedure{square\_dodd}{$\mathbf{X}$}
    \State \multiline{Perform Sinkhorn's algorithm on $\mathbf{X}^{*2}$, producing diagonal matrices $\tilde{\boldsymbol{\Lambda}}$ and $\tilde{\mathbf{M}}$, and a doubly stochastic matrix $\tilde{\mathbf{Q}}$ such that $\tilde{\boldsymbol{\Lambda}} \mathbf{X}^{*2} \tilde{\mathbf{M}} = \tilde{\mathbf{Q}}$}
    \State $\mathbf{Q} \gets \sqrt[*]{\tilde{\mathbf{Q}}}$
    \State \multiline{$\mathbf{S} \gets$ $d \times d$ matrix such that $\mathbf{S}_{ij} = 1$ if $\mathbf{X}_{ij} > 0$ and $\mathbf{S}_{ij} = -1$ if $\mathbf{X}_{ij} < 0$, $i,j \in [d]$}
    \State $\mathbf{Q} \gets \mathbf{S} \odot \mathbf{Q}$
    \State $\boldsymbol{\Lambda} \gets \left(\sqrt[*]{\tilde{\boldsymbol{\Lambda}}}\right)^{-1}$
    \State $\mathbf{M} \gets \left(\sqrt[*]{\tilde{\mathbf{M}}}\right)^{-1}$
    \State \textbf{return} $\boldsymbol{\Lambda}, \mathbf{Q}, \mathbf{M}$
\EndProcedure
\end{algorithmic}
\end{algorithm}

\subsection{The Square Case: A Procrustes-Based Algorithm}
The \textbf{Tandem Procrustes Algorithm} \cite{Everson97} generalizes the well-known \textbf{Procrustes Problem} \cite{procrustes} to orthogonal, but not necessarily orthonormal columns: Given a target matrix $\mathbf{A} \in \mathbb{R}^{m \times n}$ and a starting matrix $\mathbf{B} \in \mathbb{R}^{p \times n}$, it finds a matrix $\mathbf{U} \in \mathbb{R}^{m \times p}$ with orthogonal columns such that $||\mathbf{A}-\mathbf{UB}||^2$ is locally minimized, where $\mathbf{U} = \mathbf{V}\mathbf{D}$ with diagonal matrix $\mathbf{D} \in \mathbb{R}^{p \times p}$ and matrix with orthonormal columns $\mathbf{V} \in \mathbb{R}^{m \times p}$. We recall the Tandem Procrustes algorithm in Algorithm \ref{tandem_procrustes}.

\begin{algorithm}[h]
\caption{Tandem Procrustes Algorithm} \label{tandem_procrustes}
\textbf{Input:} 

$\mathbf{A} \in \mathbb{R}^{m \times n}$, the target matrix

$\mathbf{B} \in \mathbb{R}^{p \times n}$, the starting matrix 

\textbf{Output:} 

Diagonal matrix $\mathbf{D} \in \mathbb{R}^{p \times p}$ and matrix with orthonormal columns $V \in \mathbb{R}^{m \times p}$ such that

$||\mathbf{A}-\mathbf{VDB}||^2$ is locally minimized

\textbf{Algorithm:} 

\begin{algorithmic}[1]
\Procedure{tandem\_procrustes}{$\mathbf{A}, \mathbf{B}$}
    \State $\mathbf{D} \gets \mathbf{I}_q$ 
    \While {convergence not reached}
        \State $\mathbf{V} \gets $ orthogonal polar factor of $\mathbf{AB}^\top \mathbf{D}$
        \State $\mathbf{D}$ updated such that the $k$-th element on the diagonal $d_k = \frac{\sum_{i=1}^m [\mathbf{AB}^T]_{ik} \mathbf{V}_{ik}}{\sum_{i=1}^n \mathbf{B}_{ki}^2}, \; k \in [p]$
    \EndWhile
    \State \textbf{return} $\mathbf{V}, \mathbf{D}$
\EndProcedure
\end{algorithmic}
\end{algorithm}

We present an approximate solution to Problem \ref{dodd-problem} in the square case. In our solution, we initialize $\mathbf{Q} = \mathbf{X}$, $\boldsymbol{\Lambda} = \mathbf{M} = \mathbf{I}_n$. We first call Tandem Procrustes on $\mathbf{A} = \mathbf{Q}^\top$ and $\mathbf{B} = \mathbf{I}_n$ to find a matrix with orthogonal columns $\mathbf{VD}$ that approximates the rows of $\mathbf{Q}$: $\mathbf{Q} \approx \mathbf{DV}^\top$. We update $\boldsymbol{\Lambda}$ and $\mathbf{Q}$ by multiplying $\boldsymbol\Lambda$ with $\mathbf{D}$ on the right and $\mathbf{Q}$ with $\mathbf{D}^{-1}$ on the left, so that now $\mathbf{Q} \approx \mathbf{V}^\top$. Then we call Tandem Procrustes on $\mathbf{A} = \mathbf{Q}$ and $\mathbf{B} = \mathbf{I}_n$ to find a matrix with orthogonal columns $\mathbf{VD}$ that approximates the columns of $\mathbf{Q}$: $\mathbf{Q} \approx \mathbf{VD}$. We update $\mathbf{M}$ and $\mathbf{Q}$ by multiplying $\mathbf{M}$ with $\mathbf{D}$ on the left and $\mathbf{Q}$ with $\mathbf{D}^{-1}$ on the right, so that now $\mathbf{Q} \approx \mathbf{V}$. These alternating steps are repeated until the final $\mathbf{Q}$ found is an orthogonal matrix. Our solution algorithm is presented in Algorithm \ref{procrustes_dodd}.

\begin{algorithm}[h]
\caption{Procrustes-Based Algorithm for Square DODD} \label{procrustes_dodd}
\textbf{Input:}

$\mathbf{X} \in \mathbb{R}^{d \times d}$, a matrix with an existing DODD

\textbf{Output:} 

Diagonal matrices $\boldsymbol{\Lambda}, \mathbf{M} \in \mathbb{R}^{d \times d}$ and orthogonal matrix $\mathbf{Q} \in \mathbb{R}^{d \times d}$ such that $\mathbf{X} = \boldsymbol{\Lambda} \mathbf{Q} \mathbf{M}$

\textbf{Algorithm:} 

\begin{algorithmic}[1]
\Procedure{square\_dodd}{$\mathbf{X}$}
    \State $\boldsymbol{\Lambda} \gets \mathbf{I}_d$
    \State $\mathbf{Q} \gets \mathbf{X}$
    \State $\mathbf{M} \gets \mathbf{I}_d$
    \While {convergence not reached}
        \State $\mathbf{V}, \mathbf{D} \gets $ \textsc{tandem\_procrustes}$(\mathbf{A} = \mathbf{Q}^\top$, $\mathbf{B} = \mathbf{I}_d)$
        \State $\boldsymbol{\Lambda} \gets \boldsymbol{\Lambda} \mathbf{D}$
        \State $\mathbf{Q} \gets \mathbf{D}^{-1} \mathbf{Q}$
        \State $\mathbf{V}, \mathbf{D} \gets $ \textsc{tandem\_procrustes}$(\mathbf{A} = \mathbf{Q}$, $\mathbf{B} = \mathbf{I}_d)$
        \State $\mathbf{M} \gets \mathbf{D M}$
        \State $\mathbf{Q} \gets \mathbf{Q} \mathbf{D}^{-1}$
    \EndWhile
    \State \textbf{return} $\boldsymbol{\Lambda}$, $\mathbf{Q}$, $\mathbf{M}$
\EndProcedure
\end{algorithmic}
\end{algorithm}

\subsection{Remarks On The Square Case DODD Algorithms}

Notice that if a DODD exists for a square matrix, then the Sinkhorn-based algorithm recovers the decomposition only because the entry-wise square of an orthogonal matrix is incidentally doubly-stochastic. If one seeks to find a DODD or a DODD approximation for a square matrix that is not known to have a DODD, then the Sinkhorn-based algorithm can fail. On the other hand, the Procrustes-based algorithm finds a DODD approximation that is locally optimal in this case. We will soon see that this property of the Procrustes-based algorithm allows us to approximately solve the general case for Problem \ref{dodd-problem}, which the Sinkhorn-based algorithm cannot. A numerical and runtime comparison between the Sinkhorn- and Procrustes-based solutions for the square case is presented in Section \ref{sec:7.3}.

\subsection{The General Case}

Consider the general case: either $m \neq n$ or we do not know what value of $d \geq m,n$ gives a DODD for the $d \times d$ 0-padding of $\overline{\mathbf{X}}$. Based on numerical results in Section \ref{sec:7.3}, we conjecture that a DODD always exists for $\mathbf{X}$ in Problem $\ref{dodd-problem}$ for some sufficiently large $d \geq m,n$. Hence, suppose we already know the value of such a $d$. Then we present an approximate solution to Problem $\ref{dodd-problem}$ where we find $\boldsymbol\Lambda$, $\mathbf{Q}$, and $\mathbf{M}$. The solution involves modifying Algorithm \ref{procrustes_dodd}. We can initialize $\lambda_i = \mu_j = 1$ and the top-left $m \times n$ corner block of $\mathbf{Q}$ to be $\overline{\mathbf{X}}$. The issue we face is how to initialize the other entries of $\mathbf{Q}$, which we will call $\overline{\mathbf{Q}}$, since the corresponding entries in $\mathbf{X}$ are all $0$, and the final solution $\mathbf{Q}$ must be an orthogonal matrix:
$$\mathbf{Q} = \begin{tikzpicture}[baseline={([yshift=-.5ex]current bounding box.center)},vertex/.style={anchor=base,
    circle,fill=black!25,minimum size=18pt,inner sep=2pt}]
    \matrix [matrix of math nodes,left delimiter=(,right delimiter=)] (m)
    {
        \; & \; & \; & \; & \; & \; \\ 
        \; & \overline{\mathbf{X}} & \; & \; & \; & \; \\ 
        \; & \; & \; & \; & \; & \; \\ 
        \; & \; & \; & \; & \; & \; \\ 
        \; & \; & \; & \; & \overline{\mathbf{Q}} & \; \\ 
        \; & \; & \; & \; & \; & \; \\ 
    };
    \draw (m-3-1.south west) -- (m-3-3.south east);
    \draw (m-3-3.south east) -- (m-1-3.north east);
\end{tikzpicture} \in \mathbb{R}^{d \times d}$$
Let us call steps $6$ to $11$ in Algorithm \ref{procrustes_dodd} the \textbf{Tandem Procrustes iterations}. We address this issue by initializing all entries of $\overline{\mathbf{Q}}$ to be random numbers and then iteratively correcting these entries, replacing them with entries from an orthogonal matrix, until $\mathbf{Q}$ is itself orthogonal. At the same time, we apply the Tandem Procrustes iterations to find $\boldsymbol\Lambda$ and $\mathbf{M}$, thus solving Problem \ref{dodd-problem}. Choose a positive integer $\ell$, which is to be thought of as a ``learning rate''. First perform the Tandem Procrustes iterations $\ell$ times. Then replace the entries of $\overline{\mathbf{Q}}$ with the corresponding entries $\overline{\mathbf{V}}$ in $\mathbf{V}$, where $\mathbf{V}$ is the most recent orthogonal matrix found by the Tandem Procrustes iterations:
$$\mathbf{Q} \gets
\begin{tikzpicture}[baseline={([yshift=-.5ex]current bounding box.center)},vertex/.style={anchor=base,
    circle,fill=black!25,minimum size=18pt,inner sep=2pt}]
    \matrix [matrix of math nodes,left delimiter=(,right delimiter=)] (m)
    {
        \; & \; & \; & \; & \; & \; \\ 
        \; & (\text{diag}(\lambda_1,...,\lambda_m))^{-1}\overline{\mathbf{X}}(\text{diag}(\mu_1,...,\mu_n))^{-1} & \; & \; & \; & \; \\ 
        \; & \; & \; & \; & \; & \; \\ 
        \; & \; & \; & \; & \; & \; \\ 
        \; & \; & \; & \; & \overline{\mathbf{V}} & \; \\ 
        \; & \; & \; & \; & \; & \; \\ 
    };
    \draw (m-3-1.south west) -- (m-3-3.south east);
    \draw (m-3-3.south east) -- (m-1-3.north east);
\end{tikzpicture}$$
where $\lambda_i$ and $\mu_j$ are the most recent non-zero entries of $\boldsymbol\Lambda$ and $\mathbf{M}$.
Finally, repeat these two steps until convergence. In Section \ref{sec:7.3}, we demonstrate the numerical performance of this procedure. Also note that as in Algorithm \ref{procrustes_dodd}, one could have performed the Tandem Procrustes iterations until convergence rather than terminating after $\ell$ times. However, our numerical results suggest that choosing $\ell$ to be small and repeating these two steps greatly decreases the runtime of our solution while producing equally effective solutions. Algorithm \ref{general_dodd} presents pseudocode of this procedure.

\begin{algorithm}[h]
\caption{An Algorithm for General DODD} \label{general_dodd}
\textbf{Input:}

$\overline{\mathbf{X}} \in \mathbb{R}^{m \times n}$, a matrix

$d \geq m,n$ such that the $0$-padded matrix $\mathbf{X} \in \mathbb{R}^{d \times d}$, with $\overline{\mathbf{X}}$ in the top-left corner, admits a DODD

$\ell \geq 1$, the learning rate

\textbf{Output:} 

Diagonal matrices $\boldsymbol{\Lambda}, \mathbf{M} \in \mathbb{R}^{d \times d}$ and orthogonal matrix $\mathbf{Q} \in \mathbb{R}^{d \times d}$ such that $\mathbf{X} = \boldsymbol{\Lambda} \mathbf{Q} \mathbf{M}$

\textbf{Algorithm:} 

\begin{algorithmic}[1]
\Procedure{dodd}{$\mathbf{X},d,\ell$}
    \State $\boldsymbol{\Lambda} \in \mathbb{R}^{d \times d} \gets \text{diag}(1,..,1,0,...,0)$, where $1$ appears $m$ times
    \State $\mathbf{M} \in \mathbb{R}^{d \times d} \gets \text{diag}(1,..,1,0,...,0)$, where $1$ appears $n$ times
    \State $\mathbf{Q} \in \mathbb{R}^{d \times d} \gets$ top left $m \times n$ corner block is $\overline{\mathbf{X}}$ and are random numbers otherwise
    \While {convergence not reached}
        \For {$i = 1,...,\ell$}
            \State $\mathbf{V}, \mathbf{D}\gets $ Tandem\_Procrustes$(\mathbf{A} = \mathbf{Q}^\top$, $\mathbf{B} = \mathbf{I}_d)$
            \State $\boldsymbol{\Lambda} \gets \boldsymbol{\Lambda} \mathbf{D}$
            \State $\mathbf{Q} \gets \mathbf{D}^{-1} \mathbf{Q}$
            \State $\mathbf{V}, \mathbf{D} \gets $ Tandem\_Procrustes$(\mathbf{A} = \mathbf{Q}$, $\mathbf{B} = \mathbf{I}_d)$
            \State $\mathbf{M} \gets \mathbf{D M}$
            \State $\mathbf{Q} \gets \mathbf{Q} \mathbf{D}^{-1}$
        \EndFor
        \State $\mathbf{Q} \gets$ replace last $d-m$ columns and $d-n$ rows of $\mathbf{Q}$ with corresponding elements in $\mathbf{V}$
    \EndWhile
    \State \textbf{return} $\boldsymbol{\Lambda}, \mathbf{Q}, \mathbf{M}$
\EndProcedure
\end{algorithmic}
\end{algorithm}

\medskip

\section{Numerical Results}
\label{sec:7}

{{

The results for symmetric trains of length 2 and the DODD were computed using Python 3.7.6 and the TensorLy package \cite{tensorly} on a Dual-Core 1.8 GHz Intel i5 Processor, while the results for longer symmetric orthogonal trains were computed in MATLAB 2020 and the Tensor Toolbox for MATLAB package \cite{TTB_Software} on a Quad-Core 2.3 GHz Intel i5 Processor. All our code can be found at \url{https://github.com/karimhalaseh/Tensor-Network-Decompositions}.

}}

\begin{table}[t]
    \centering
    \begin{tabular}{|c|c|c|c|c|c|c|}
        \hline
        \multicolumn{5}{|c|}{Controls} & \multicolumn{2}{|c|}{Results}  \\
        \hline
        $n$ & $r_\mathcal{A}$ & $r_\mathcal{B}$ & Orthogonal? & $\sigma$ & PSD $\mathbf{C}_\mathcal{A}$, $\mathbf{C}_\mathcal{B}$? & $10^{\text{Avg. } \log_{10} \text{ Rel. Error}}$ \\
        \hline
        \multirow{3}{*}{5} & \multirow{3}{*}{2} & \multirow{3}{*}{3} & \multirow{3}{*}{Yes} & 0 & - & $2.560 \times 10^{-15}$ \\
        & & & & $10^{-6}$ & - & $9.584 \times 10^{-5}$\\
        & & & & $10^{-2}$ & - & 0.0726 \\
        \hline
        \multirow{3}{*}{5} & \multirow{3}{*}{5} & \multirow{3}{*}{5} & \multirow{3}{*}{Yes} & 0 & - & $8.928 \times 10^{-15}$ \\
        & & & & $10^{-6}$ & - & $1.607 \times 10^{-4}$ \\
        & & & & $10^{-2}$ & - & 0.113 \\
        \hline
        \multirow{3}{*}{25} & \multirow{3}{*}{7} & \multirow{3}{*}{4} & \multirow{3}{*}{Yes} & 0 & - & $1.685 \times 10^{-14}$ \\
        & & & & $10^{-6}$ & - & $2.019 \times 10^{-4}$ \\
        & & & & $10^{-2}$ & - & 0.5762 \\
        \hline
        \multirow{3}{*}{5} & \multirow{3}{*}{2} & \multirow{3}{*}{3} & \multirow{3}{*}{No} & 0 & 98 & $8.175 \times 10^{-15}$ \\
        & & & & $10^{-6}$ & 9 & 0.9999 \\
        & & & & $10^{-2}$ & 1 & 1.0003 \\
        \hline
        \multirow{3}{*}{5} & \multirow{3}{*}{5} & \multirow{3}{*}{5} & \multirow{3}{*}{No} & 0 & 97 & $6.622 \times 10^{-12}$ \\
        & & & & $10^{-6}$ & 96 & 0.2394 \\
        & & & & $10^{-2}$ & 76 & 1.3259 \\
        \hline
        \multirow{3}{*}{25} & \multirow{3}{*}{7} & \multirow{3}{*}{4} & \multirow{3}{*}{No} & 0 & 100 & $1.299 \times 10^{-14}$ \\
        & & & & $10^{-6}$ & 0 & - \\
        & & & & $10^{-2}$ & 0 & - \\
        \hline
    \end{tabular}
    \caption{Numerical results. For each size permutation, 100 tests are run.}
    \label{symmL2_results}
\end{table}

\subsection{Symmetric Orthogonal Tensor Trains of Length $2$}{\label{sym-len-2-test}}
Here we tested Algorithm \ref{symmL2_alg}. The parameters $n$, $r_\mathcal{A}$, $r_\mathcal{B}$ (cf. Section \ref{sec:3}), control the size of Problem \ref{sym-len-2-prob}, and for each choice of these parameters, 100 tests were conducted. In the same approach to testing as Kolda's slice method \cite{kolda2015symmetric}, a tensor $\mathcal{T}^*$ with an exact decomposition is artificially constructed. Then a new tensor $\mathcal{T}$ is formed by corrupting $\mathcal{T}^*$ with Gaussian noise:
\begin{equation*}
\mathcal{T} = \mathcal{T}^* + \sigma \frac{||\mathcal{T}^*||}{||\mathcal{N}||} \mathcal{N}
\end{equation*}
where $\sigma$ is a noise parameter and $\mathcal{N}$ is a tensor of the same dimensions as $\mathcal{T}^*$ whose entries are independently sampled, standard normally distributed numbers. We tested for $\sigma \in \{0, 10^{-2}, 10^{-6}\}$. If whitening is applied, then the algorithm was given 200 iterations to find PSD matrices $\mathbf{C}_\mathcal{A}$ and $\mathbf{C}_\mathcal{B}$ before declaring failure. The relative error between the solution returned by the algorithm $\widehat{\mathcal{T}}$ and $\mathcal{T}$ was then computed:
$$\text{relative error} = \frac{\|\widehat{\mathcal{T}} - \mathcal{T}\|}{\|\mathcal{T}\|}$$
For a successful solution, we expect the average $\log_{10}$ relative error when raised to the power of $10$ to be on the order of $\sigma$. The results are listed in Table \ref{symmL2_results}. We see that with no noise, the algorithm in both the orthogonal and non-orthogonal cases performed well, with PSD matrices found in most tests for whitening. In the orthogonal case when noise is added, the relative error was generally one order of magnitude more than the noise, similar to Kolda's results in~\cite{kolda2015symmetric}. As in Kolda's tests for the slice method, the poorest performance was for noise in the non-orthogonal case. Kolda hypothesizes that this is due to the noisy tensor having a rank higher than the rank found by the algorithm \cite{kolda2015symmetric}.

\begin{table}[t]
    \centering
    \begin{tabular}{|c|c|c|c|c|c|}
        \hline
        \multicolumn{4}{|c|}{Controls} & \multicolumn{2}{|c|}{Results}  \\
        \hline
        $n$ & $L$ & $r_{\mathcal{X}_i}$ & $\sigma$ & $10^{\text{Avg. } \log_{10} \text{ Rel. Error}}$ & Avg. Runtime (s) \\
        \hline
        \multirow{3}{*}{4} & \multirow{3}{*}{3} & \multirow{3}{*}{2,2,2} & 0 & $4.282 \times 10^{-14}$ & 2.753 \\
        & & & $10^{-6}$ & $2.077 \times 10^{-3}$ & $2.337$ \\
        & & & $10^{-2}$ & $0.9453$ & $0.5653$ \\
        \hline
        \multirow{3}{*}{4} & \multirow{3}{*}{3} & \multirow{3}{*}{4,4,4} & 0 & $2.432 \times 10^{-13}$ & 0.2475 \\
        & & & $10^{-6}$ & $9.331 \times 10^{-4}$ & 0.2176 \\
        & & & $10^{-2}$ & 0.7622 & 0.2641 \\
        \hline
        \multirow{3}{*}{4} & \multirow{3}{*}{6} & \multirow{3}{*}{2,2,2,2,2,2} & 0 & $1.346 \times 10^{-12}$ & 4.784 \\
        & & & $10^{-6}$ & 0.1087 & 5.340 \\
        & & & $10^{-2}$ & 1.282 & 10.275 \\
        \hline
        \multirow{3}{*}{8} & \multirow{3}{*}{3} & \multirow{3}{*}{2,2,2} & 0 & $7.424 \times 10^{-14}$ & 13.123 \\
        & & & $10^{-6}$ & $5.826 \times 10^{-3}$ & 9.667 \\
        & & & $10^{-2}$ & 1.128 & 1.210 \\
        \hline
    \end{tabular}
    \caption{Numerical results. For each size and noise permutation, 100 tests are run.}
    \label{symm-orth-3-table}
\end{table}

\subsection{Symmetric Orthogonal Tensor Trains of Length $L \geq 3$}

We tested Algorithm \ref{symm-odeco-alg} on tensor trains satisfying the DRC for varying $n$, $L$, and ranks of carriages $r_{\mathcal{X}_i}$, where $i \in [L]$ (cf. Section~\ref{sec:4}). The method of testing was identical to that of Section \ref{sym-len-2-test}. The results are presented in Table \ref{symm-orth-3-table}. Without noise, the algorithm was successful, but even in the presence of small noise, the relative error was high, likely due to the rank of the tensor increasing, as Kolda hypothesized. The runtimes were also diverse due to dependence on $r_{\mathcal{X}_i}$ (since kernel completion will take longer for lower ranks) and stopping criteria whenever the rank increases.

\begin{figure}[t]
    \centering
    \begin{subfigure}{0.49\linewidth}
        \centering
        \includegraphics[scale = 0.5]{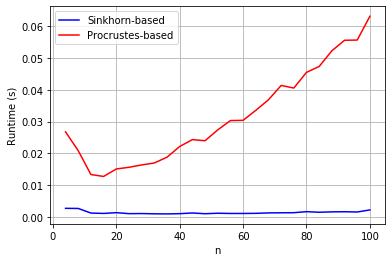}
        \caption{Runtime Comparison}
        \label{runtime}
    \end{subfigure}
    \begin{subfigure}{0.49\linewidth}
        \centering
        \includegraphics[scale = 0.5]{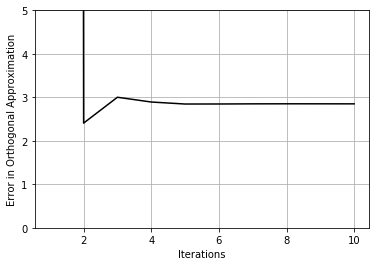}
        \caption{Performance of Procrustes-based approximation}
        \label{procrustes_approx}
    \end{subfigure}
    \caption{(a): A runtime comparison of the two square DODD algorithms; (b): An example of the error output of each Tandem Procrustes call when Procrustes-based DODD is run on a random matrix $(n = 10)$}
    \label{runtime_and_approx}
\end{figure}

\begin{figure}[h]
    \centering
    \begin{subfigure}{0.49\linewidth}
        \centering
        \includegraphics[scale = 0.525]{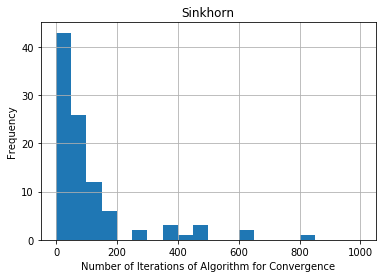}
        \caption{Sinkhorn-based DODD}
        \label{iters_sh_n3}
    \end{subfigure}
    \begin{subfigure}{0.49\linewidth}
        \centering
        \includegraphics[scale = 0.525]{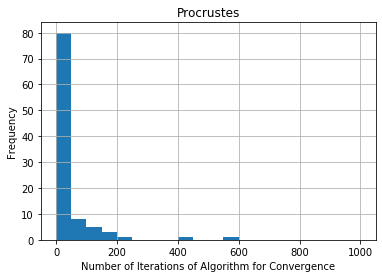}
        \caption{Procrustes-based DODD}
        \label{iters_pc_n3}
    \end{subfigure}
    \caption{Number of iterations until convergence for DODD algorithms for $n = 3$ (100 tests run)}
    \label{iters_n3}
\end{figure}

\begin{figure}[h]
    \centering
    \begin{subfigure}{0.49\linewidth}
        \centering
        \includegraphics[scale = 0.5]{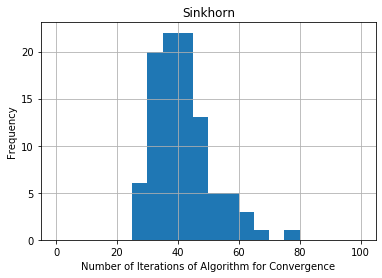}
        \caption{Sinkhorn-based DODD}
        \label{iters_sh_n10}
    \end{subfigure}
    \begin{subfigure}{0.49\linewidth}
        \centering
        \includegraphics[scale = 0.5]{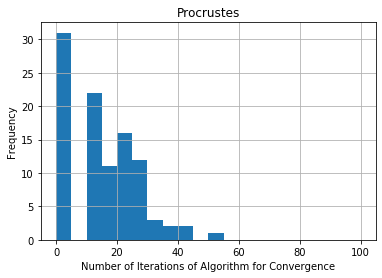}
        \caption{Procrustes-based DODD}
        \label{iters_pc_n10}
    \end{subfigure}
    \caption{Number of iterations until convergence for DODD algorithms for $n = 10$ (100 tests run)}
    \label{iters_n10}
\end{figure}

\subsection{Matrix Diagonal-Orthogonal-Diagonal Decomposition}\label{sec:7.3}

Next we tested the Sinkhorn-based Algorithm \ref{sinkhorn_dodd} and Procrustes-based Algorithm \ref{procrustes_dodd} for the square case of DODD. We constructed 100 square matrices with exact decompositions for sizes $n \in \{3,10,25\}$. For Algorithm \ref{sinkhorn_dodd}, convergence was said to be achieved when the sum of squares of differences between the row and column sums of $\mathbf{Q}^{*2}$ and $1$ was less than $10^{-28}$. For Algorithm \ref{procrustes_dodd}, the learning rate $\ell$ was set to 1, and convergence was said to be achieved when the most recent call to the Tandem Procrustes algorithm satisfied $||\mathbf{Q}-\mathbf{VD}||^2 < 10^{-28}$. Both algorithms were given a maximum of 1000 iterations before terminating if convergence was not achieved before then. Both algorithms were said to be successful if the relative error of $\mathbf{QQ}^\top$ and $\mathbf{I}_n$ was within $10^{-10}$. For all values of $n$ and all tests, both algorithms converged and were successful. Figures \ref{iters_n3} and \ref{iters_n10} show histograms of the number of iterations before convergence for each algorithm. We saw that as $n$ increased, the number of such iterations generally decreased for both algorithms. The Procrustes-based algorithm generally required fewer iterations than the Sinkhorn-based algorithm. However, as shown in Figure \ref{runtime}, the overall runtime for Algorithm \ref{sinkhorn_dodd} was faster than that of Algorithm \ref{procrustes_dodd}, since the Tandem Procrustes algorithm requires computing polar factors. Figure \ref{procrustes_approx} shows a sample of the error trajectory of the Tandem Procrustes algorithm on a random $10 \times 10$ matrix. Even for a matrix unlikely to have a square DODD, the algorithm appears to converge to a local minimum.

\begin{table}[t]
    \centering
    \begin{tabular}{|c|c|c|c|c|c|c|}
        \hline
        \multicolumn{4}{|c|}{Controls} & \multicolumn{3}{|c|}{Results}  \\
        \hline
        $m$ & $n$ & $d$ & Exact? & Rel. Error $< 10^{-10}$? & 10\^~(Avg. $\log_{10}$ Rel. Error) & Avg. Runtime $(s)$ \\
        \hline
        6 & 5 & 6 & Yes & 73 & $1.02 \times 10^{-11}$  & 1.000 \\
        6 & 5 & 6 & No & 0 & 0.769 & 1.310 \\
        \hline
        6 & 5 & 7 & Yes & 0 & $2.65 \times 10^{-3}$  & 1.246 \\
        6 & 5 & 7 & No & 0 & 0.230 & 1.277 \\
        \hline
        6 & 5 & 8 & Yes & 16 & $6.07 \times 10^{-7}$  & 1.266 \\
        6 & 5 & 8 & No & 40 & $1.30 \times 10^{-8}$ & 1.181 \\
        \hline
        6 & 5 & 9 & Yes & 54 & $3.03 \times 10^{-10}$  & 1.211 \\
        6 & 5 & 9 & No & 86 & $9.35 \times 10^{-13}$ & 1.348 \\
        \hline
        6 & 5 & 10 & Yes & 64 & $4.42 \times 10^{-11}$  & 1.246 \\
        6 & 5 & 10 & No & 85 & $7.89 \times 10^{-13}$ & 1.024 \\
        \hline
        6 & 5 & 15 & Yes & 97 & $8.57 \times 10^{-14}$  & 1.008 \\
        6 & 5 & 15 & No & 96 & $8.73 \times 10^{-14}$ & 1.035 \\
        \hline
        6 & 5 & 30 & Yes & 100 & $2.47 \times 10^{-14}$  & 1.071 \\
        6 & 5 & 30 & No & 100 & $2.68 \times 10^{-14}$ & 1.208 \\
        \hline
    \end{tabular}
    \caption{General case DODD for both the ``exact" and ``non-exact" case. For each permutation of control parameters, 100 tests are run, with each test given a maximum of 1000 iterations to converge.}
    \label{general_dodd_results}
\end{table}

Lastly, we examined Algorithm \ref{general_dodd} for the general case of DODD. For $m = 6$, $n = 5$, and each value of $d \geq m,n$ tested, we constructed 100 $m \times n$ matrices whose 0-padded $d \times d$ matrix admitted exact decompositions (``exact''), and 100 $m \times n$ matrices whose entries are independently sampled from a normal distribution with mean 0 and variance $5^2$ (``non-exact''). We then ran Algorithm \ref{general_dodd} on these matrices, inputting $d$ for both exact and non-exact tests, with learning rate $\ell = 2$. The convergence and success criteria was the same as for the square case. The results are presented in Table \ref{general_dodd_results}. We saw that when $d = \max\{m, n\}$, the algorithm mostly performed well for the exact tests and poorly performed for the non-exact tests. What was intriguing, however, was that for $d = \{m,n\} + 1$, we observed no successes for both tests, and as $d$ increased, the frequency of success also increased, reaching complete success even for the non-exact tests when $d$ was sufficiently large. This lead us to the following conjecture:

\begin{conjecture}{\label{dodd-conjecture}}
The $d \times d$ 0-padding of any matrix $\mathbf{X} \in \mathbb{R}^{m \times n}$ admits a DODD, for some $d \geq m,n$ sufficiently large. This means that every matrix $\mathbf{X} \in \mathbb{R}^{m \times n}$ has a decomposition of the form $\mathbf{X} = \boldsymbol\Lambda \overline{\mathbf{Q}}\mathbf{M}$, where $\boldsymbol\Lambda \in \mathbb{R}^{m \times m}$ and $\mathbf{M} \in \mathbb{R}^{n \times n}$ are diagonal matrices and $\overline{\mathbf{Q}} \in \mathbb{R}^{m \times n}$ is the $m \times n$ top-left corner block of an orthogonal matrix $\mathbf{Q} \in \mathbb{R}^{d \times d}$. Equivalently, $\mathbf{X}$ has a decomposition of the form $\mathbf{X} = \boldsymbol\Lambda\mathbf{C}^\top\mathbf{F}\mathbf{M}$, where $\mathbf{C} \in \mathbb{R}^{m \times d}$ is a matrix with orthonormal rows and $\mathbf{F} \in \mathbb{R}^{d \times n}$ is a matrix with orthonormal columns (we can choose $\mathbf{C}^\top$ to be the first $m$ rows of $\mathbf{Q}$ and $\mathbf{F}$ to have the first $n$ standard basis vectors in $\mathbb{R}^d$ as columns).
\end{conjecture}

\section{Conclusion and Future Work}
\label{sec:8}

In this paper, we studied the decomposition of tensors into tensor trains whose carriages are symmetric or orthogonally decomposable tensors. For length-2 trains, we showed that we can find the decomposition when the carriages are symmetric and odeco (Section~\ref{sec:3.1}), only symmetric (Section~\ref{sec:3.2}), or only odeco (Section~\ref{sec:5}). For longer length tensor trains, we studied the case when the carriages are both symmetric and odeco. For such networks, we provided algorithms for decomposition which require the Decreasing Ranks Condition (Section~\ref{sec:4}). In Section~\ref{sec:7}, we provided numerical results that support our findings.

A variety of open problems arose during our study. We showed that Kolda's whitening procedure for symmetric tensors could be adapted to the similar case of a tensor train of length 2, whereby a tensor with a symmetric decomposition could be transformed into an orthogonal problem. It remains an open problem to find an efficient equivalent to longer tensor trains, and so encompass a greater family of tensors for such a decomposition. 

\begin{problem}
Can we find a whitening procedure for symmetrically decomposable tensor trains of length greater than 2?
\end{problem}

One of the interesting mathematical problems that arose during our study was that of finding a Diagonal-Orthogonal-Diagonal Decomposition for matrices (Section~\ref{sec:6}). We are not aware of any linear algebra literature that has previously studied this topic, and there are a number of issues to still address.

\begin{problem}
For what hidden dimension $d$ can we decompose any given matrix via a Diagonal-Orthogonal-Diagonal Decomposition?
\end{problem}

We saw via our numerical tests that taking $d$ to be the maximum of the two dimensions of a matrix and applying our algorithm did not result in an exact DODD, but rather convergence to a suboptimal approximation. When $d$, however, was strictly greater than both dimensions, we saw that the error output at each iteration of our algorithm continuously decreased, and given either a sufficiently large maximum threshold of iterations, or given a sufficiently large $d$ for efficient convergence, an exact solution was generally achieved. It remains an open problem to find exactly the minimum such hidden dimension $d$ relative to the dimensions of a given matrix such that an exact DODD exists.

Assuming an exact solution exists, we solved the square case for the DODD, i.e., the case where the hidden dimension $d$ is equal to both dimensions of the given matrix. The question remains for exactly what family of matrices this is the case.

\begin{problem}
Describe the set of matrices (as a set cut out by polynomial equations) that have a square Diagonal-Orthogonal-Diagonal Decomposition.
\end{problem}

More broadly, we began our study with a focus on the decomposition of orthogonal tensor networks in general. As a specific network form with many modern applications, and an ability to describe large families of tensors efficiently, tensor trains were the primary case we considered, and due to fruitful progress resulted in being the main focus of this paper. However, we remain interested in other, more general tensor networks.

\begin{problem}
Can we decompose orthogonal tensor networks for any general tensor network? For example, can we find an orthogonal tensor ring decomposition~\cite{Jianfeng}? Can our methods for orthogonal tensor trains be extended to orthogonal tensor trees?
\end{problem}

The primary motivation for this work is to find a structured decomposition that applies to all tensors. Orthogonal tensor train decompositions do not apply to all tensors of a given size. As discussed earlier, the rank of a tensor decomposing this way can be up to $\mathcal O(n^{d-2})$ while the generic rank of a dimension-$n$ order-$d$ tensor is $\mathcal O(n^{d-1})$. Orthogonality will, of course, make the orthogonal tensor train decomposition even more restrictive. So, a natural question is which orthogonal tensor networks can represent all tensors of a given size. Another direction for future work is to devise a method which approximately decomposes a given tensor as a tensor train.

\begin{problem}
What are the simplest orthogonal tensor networks according to which any tensor of a given size can be decomposed? Given a tensor and a tensor network diagram, can we approximately decompose the tensor according to that diagram?
\end{problem}

Along the lines of the previous question, it would be interesting to give an implicit description of the set of tensors that decompose according to a given orthogonal tensor network.

\begin{problem}
Find the polynomial equations that define the set of tensors that decompose according to a given orthogonal tensor network. Is the set of such tensors Zariski closed?
\end{problem}

Finding the eigenvectors and singular vector tuples of a general tensor is an NP-hard problem~\cite{HL}. However, the problem is easy for the family of orthogonally decomposable tensors~\cite{Robeva, RobSei}. How about for orthogonal tensor networks?

\begin{problem}
Given a tensor decomposing according to an orthogonal tensor network, can we find its singular vector tuples efficiently? Can we describe the set of all of its singular vector tuples in terms of its orthogonal tensor network decomposition?
\end{problem}

Answering these questions would provide fundamental progress towards finding a structured tensor decomposition for any tensor, analogous to the singular value decomposition for matrices.

\section*{Acknowledgements}  KH was supported  by a summer WLIURA grant, TM was supported by an NSERC USRA summer grant, and ER was supported by an NSERC  Discovery Grant (DGECR-2020-00338).

\bibliographystyle{ieeetr}
\bibliography{sample.bib}

\begin{thebibliography}{10}

\bibitem{Hitchcock}
F.~L. Hitchcock, ``The expression of a tensor or a polyadic as a sum of
  products,'' {\em Journal of Mathematics and Physics}, vol.~6, pp.~164--189,
  September 1927.

\bibitem{HL}
C.~Hillar and L.-H. Lim, ``Most tensor problems are np-hard,'' {\em Journal of
  the ACM}, vol.~60, no.~6, 2013.

\bibitem{KoBa09}
T.~G. Kolda and B.~W. Bader, ``Tensor decompositions and applications,'' {\em
  SIAM Review}, vol.~51, pp.~455--500, Sep 2009.

\bibitem{kolda2015symmetric}
T.~G. Kolda, ``Symmetric orthogonal tensor decomposition is trivial,'' 2015.

\bibitem{Anandkumar2014}
A.~Anandkumar, R.~Ge, S.~K. D.~Hsu, and M.~Telgarsky, ``Tensor decompositions
  for learning latent variable models,'' {\em Journal of Machine Learning
  Research}, vol.~15, no.~80, pp.~2773--2832, 2014.

\bibitem{BDHR}
A.~Boralevi, J.~Draisma, E.~Horobet, , and E.~Robeva, ``Orthogonal and unitary
  tensor decomposition from an algebraic perspective,'' {\em Israel Journal of
  Mathematics}, vol.~222, no.~1, p.~223–260, 2017.

\bibitem{Robeva}
E.~Robeva, ``Orthogonal decomposition of symmetric tensors,'' {\em SIAM Journal
  on Matrix Analysis and Applications}, vol.~37, no.~1, pp.~86--102, 2016.

\bibitem{RobSei}
E.~Robeva and A.~Seigal, ``Singular vectors of orthogonally decomposable
  tensors,'' {\em Linear and Multilinear Algebra}, vol.~65, no.~12,
  pp.~2457--2471, 2017.

\bibitem{KP}
J.~Kileel and J.~Pereira, ``Subspace power method for symmetric tensor
  decomposition and generalized pca,'' {\em Preprint: arXiv:1912.04007}, 2019.

\bibitem{quantum}
J.~Bridgeman, ``Hand-waving and interpretive dance: An introductory course on
  tensor networks,'' {\em Journal of Physics A Mathematical and Theoretical},
  vol.~50, no.~22, 2016.

\bibitem{TNNutshell}
J.~Biamonte and V.~Bergholm, ``Tensor networks in a nutshell,'' {\em preprint
  arXiv:1708.00006}, 2017.

\bibitem{oseledets}
I.~V. Oseledets, ``Tensor-train decomposition,'' {\em SIAM J. Sci. Comput.},
  vol.~33, p.~2295–2317, Sep 2011.

\bibitem{RobSei2}
E.~Robeva and A.~Seigal, ``Duality of tensor networks and graphical models,''
  {\em Information and Inference: A Journal of the IMA}, vol.~8, no.~2,
  p.~273–288, 2019.

\bibitem{sinkhorn1967}
R.~Sinkhorn and P.~Knopp, ``Concerning nonnegative matrices and doubly
  stochastic matrices.,'' {\em Pacific J. Math.}, vol.~21, no.~2, pp.~343--348,
  1967.

\bibitem{Everson97}
R.~Everson, ``Orthogonal, but not orthonormal, procrustes problems,'' in {\em
  Advances in Computational Mathematics . (Submitted). Available from
  http://www.ee.ic.ac.uk/research/neural/everson}, 1997.

\bibitem{Jianfeng}
Z.~Chen, Y.~Li, and J.~Lu, ``Tensor ring decomposition: Energy landscape and
  one-loop convergence of alternating least squares,'' {\em SIAM J. Matrix
  Anal. Appl.; preprint arXiv:1905.07101}, 2019.

\bibitem{hohenstein2012}
E.~G. {Hohenstein}, R.~M. {Parrish}, and T.~J. {Mart{\'\i}nez}, ``{Tensor
  hypercontraction density fitting. I. Quartic scaling second- and third-order
  M{\o}ller-Plesset perturbation theory},'' {\em The Journal of Chemical
  Physics}, vol.~137, pp.~044103--044103, July 2012.

\bibitem{parrish2012}
R.~M. Parrish, E.~G. Hohenstein, T.~J. Martínez, and C.~D. Sherrill, ``Tensor
  hypercontraction. ii. least-squares renormalization,'' {\em The Journal of
  Chemical Physics}, vol.~137, no.~22, p.~224106, 2012.

\bibitem{lu2015compression}
J.~Lu and L.~Ying, ``Compression of the electron repulsion integral tensor in
  tensor hypercontraction format with cubic scaling cost,'' 2015.

\bibitem{procrustes}
P.~Schönemann, ``A generalized solution of the orthogonal procrustes
  problem,'' {\em Psychometrika}, vol.~31, p.~1–10, 1966.

\bibitem{tensorly}
J.~Kossaifi, Y.~Panagakis, A.~Anandkumar, and M.~Pantic, ``Tensorly: Tensor
  learning in python,'' {\em CoRR}, 2018.

\bibitem{TTB_Software}
B.~W. Bader, T.~G. Kolda, {\em et~al.}, ``Matlab tensor toolbox version 3.1.''
  Available online, June 2019.

\end{thebibliography}

\end{document}